\documentclass[11pt]{article}
\usepackage{a4wide}
\usepackage {times,epic,eepic,graphicx}
\usepackage {amsthm,amsmath,amssymb,amsfonts,amscd}
\usepackage {dsfont}
\usepackage {alltt}
\usepackage[utf8]{inputenc}
\usepackage[english]{babel}
\usepackage[english]{isodate}
\usepackage[parfill]{parskip}

\usepackage {graphicx,epsfig,psfrag}
\usepackage {algorithm}
\usepackage {algorithmic}
\usepackage {graphicx}
\usepackage {multirow}
\usepackage [figuresright]{rotating}
\usepackage {rotating}
\usepackage {subfigure}
\usepackage {epsfig}
\usepackage {url}
\usepackage{todonotes}

    \renewcommand  {\phi}    { \varphi }

    \newcommand    {\ten}    { \otimes }
    \newcommand    {\Ten}    { \bigotimes }

    \newcommand    {\bq}   { \begin{equation} }
    \newcommand    {\eq}   { \end{equation} }
    \newcommand    {\ba}   { \begin{array} }
    \newcommand    {\ea}   { \end{array} }
    \newcommand    {\BbbN}  {\mathds{N}}
    \newcommand    {\N}     {\BbbN}

    \newcommand    {\BbbR}  {\mathds{R}}
    \newcommand    {\R}     {\BbbR}

    \newcommand    {\Id}    { \mbox{\bf Id} }

    \newcommand    {\Bild}  { \mbox{\rm Range$\,$} }

    \newcommand    {\Diag}  { \mbox{\rm diag$\,$} }

    \newcommand    {\Span}  { \mbox{\rm span} }
    \newcommand    {\argmin}{ \mbox{\rm argmin} }

\newcommand    {\dotprod}[2] { \left \langle #1, #2 \right \rangle }

\newcommand    {\ve}[1]
                {\underline{#1}}

\newcommand    {\dist}[2]
                {\mbox{\rm dist$\,$} \left ( #1, #2 \right ) }

\newcommand    {\Ap}
                {\mathcal{A}}

\newcommand    {\Mp}
                {\mathcal{M}}

\newcommand    {\Vp}
                {\mathcal{V}}

\newcommand{\BIGOP}[1]{\mathop{\mathchoice%
{\raise-0.22em\hbox{\huge $#1$}} {\raise-0.05em\hbox{\Large $#1$}}
{\hbox{\large $#1$}}{#1}}}

\newcommand{\BIGboxplus}{\mathop{\mathchoice%
{\raise-0.35em\hbox{\huge $\boxplus$}}%
{\raise-0.15em\hbox{\Large $\boxplus$}}{\hbox{\large
$\boxplus$}}{\boxplus}}}
\newcommand{\bigtimes}{\BIGOP{\times}}

\newtheorem{theorem}{Theorem}[section]
\newtheorem{remark}[theorem]{Remark}
\newtheorem{defn}[theorem]{Definition}
\newtheorem{note}[theorem]{Notation}

\newtheorem{example}[theorem]{Example}

\newtheorem{lemma}[theorem]{Lemma}

\newtheorem{corollary}[theorem]{Corollary}

\newtheorem{proposition}[theorem]{Proposition}

\newlength{\symboheight}
\graphicspath{{images/}}
\ifx\undefined\pdfpageheight
  \def\image#1#2{\epsfig{file=#1,width=#2}}
\else
  \def\image#1#2{\epsfig{file=#1,width=#2}}
\fi

\begin{document}
\title{Convergence of Alternating Least Squares Optimisation for Rank-One Approximation to High Order Tensors}
\author{
  Mike Espig
     \thanks{RWTH Aachen University, Germany}
       \footnote{
     Address: RWTH Aachen University, Department of Mathematics, IGPM
Templergraben 55, 52056 Aachen Germany.
     Phone: +49 (0)241 80 96343, E-mail address: mike.espig@alopax.de
   }
  \and
  Aram Khachatryan
  \footnotemark[1]
  }
\maketitle

\begin{abstract}

The approximation of tensors has important applications in various
disciplines, but it remains an extremely challenging task. It is
well known that tensors of higher order can fail to have best
low-rank approximations, but with an important  exception that best
rank-one approximations always exists. The most popular approach to
low-rank approximation is the alternating least squares (ALS)
method. The convergence of the alternating least squares algorithm
for the rank-one approximation problem is analysed in this paper. In
our analysis we are focusing on the global convergence and the rate
of convergence of the ALS algorithm. It is shown that the ALS method
can converge sublinearly, Q-linearly, and even Q-superlinearly. Our
theoretical results are illustrated on explicit examples.

\end{abstract}
{\bf Keywords}:  tensor format, tensor representation, alternating
least squares optimisation, orthogonal projection method.

{\bf MSC}:  15A69, 49M20, 65K05, 68W25, 90C26.


\section{Introduction}\label{sec:introduction}
We consider a minimisation problem on the tensor space
$\Vp=\Ten_{\mu=1}^d \R^{n_\mu}$ equipped with the Euclidean inner
product $\dotprod{\cdot}{\cdot}$. The objective function $f: \Vp \rightarrow \R$
of the optimisation task is quadratic
\begin{equation}\label{equ:deff}
f(v):=\frac{1}{\|b\|^2}\left[\frac{1}{2} \dotprod{v}{v} -
\dotprod{b}{v} \right] \geq -\frac{1}{2},
\end{equation}
where $b \in \Vp$. In our analysis, a tensor $u\in \Vp$ is
represented as a rank-one tensor. The representation of rank-one
tensors is described by the following multilinear map $U$:
\begin{eqnarray*}
  U : P:=\bigtimes_{\mu=1}^d \R^{n_\mu}&\rightarrow &  \Vp\\
  (p_1, \dots, p_d) &\mapsto& U(p_1,\dots, p_d):=\Ten_{\mu=1}^d
  p_\mu.
\end{eqnarray*}
We call a $d$-tuple of vectors $(p_1, \dots, p_d)\in P$ a
representation system of $u$ if $u=U(p_1, \dots, p_d)$.
The tensor $b$ is approximated with respect to rank-one tensors,
i.e. we are looking for a representation system $(p^*_1, \dots,
p^*_d)\in P$ such that for
\begin{eqnarray}\label{equ:defF}
    F&:=&f \circ U :P \rightarrow \Vp \rightarrow \R\\
    \nonumber
    F(p_1, \dots, p_d)&=&\frac{1}{\|b\|^2}\left[\frac{1}{2} \dotprod{U(p_1, \dots, p_d)}{U(p_1, \dots, p_d)} -
\dotprod{b}{U(p_1, \dots, p_d)} \right]
\end{eqnarray}
we have
\begin{equation}\label{eq:defMinProblem}
F(p^*_1, \dots, p^*_d) = \min_{(p_1, \dots, p_d) \in P} F(p_1,
\dots, p_d).
\end{equation}
The range set $U(P)$ is a closed in $\Vp$, see \cite{HA12}.
Therefore, the approximation problem is well defined. The set of
best rank-one approximations of the tensor $b$ is denoted by
\begin{equation}\label{equ:defMb}
    \Mp_b:=\left\{v \in U(P) \,: \, v \mbox{ is a best rank-one approximation of } b\right\}.
\end{equation}
The alternating least squares (ALS) algorithm \cite{BEMO02, BEMO05,
ESHAHARS11_2, HoltzALS2012, Kolda09tensordecompositions,
Oseledets2011, OseledetsDolgov2012} is recursively defined. Suppose
that the $k$-th iterate $\ve{p}^k=(p_1^k, \dots, p_d^k)$ and the
first $\mu-1$ components $p_1^{k+1}, \dots, p_{\mu-1}^{k+1}$ of the
$(k+1)$-th iterate $\ve{p}^{k+1}$ have been determined. The basic
step of the ALS algorithm is to compute the minimum norm solution
\begin{equation*}
    p_\mu^{k+1}:=\argmin_{q_\mu \in \R^{n_\mu}}F(p_1^{k+1}, \dots, p_{\mu-1}^{k+1}, q_\mu, p_{\mu+1}^{k}, \dots,
    p_{d}^{k}).
\end{equation*}
Thus, in order to obtain $\ve{p}^{k+1}$ from $\ve{p}^k$, we have to
solve successively $L$ ordinary least squares problems.
The ALS algorithm is a nonlinear Gauss-Seidel method. The locale
convergence of the nonlinear Gauss-Seidel method to a stationary
point $\ve{p}^* \in P$ follows from the convergence of the linear
Gauss-Seidel method applied to the Hessian $F''(\ve {p}^*)$ at the
limit point $\ve {p}^*$. If the linear Gauss-Seidel method converges
R-linear then there exists a neighbourhood $B(\ve {p}^*)$ of $\ve
{p}^*$ such that for every initial guess $\ve{p}^0 \in B(\ve {p}^*)$
the nonlinear Gauss-Seidel method converges R-linear with the same
rate as the linear Gauss-Seidel method. We refer the reader to
Ortega and Rheinboldt for a description of nonlinear Gauss-Seidel
method \cite[Section 7.4]{OR70} and convergence analysis \cite[Thm.
10.3.5, Thm. 10.3.4, and Thm. 10.1.3]{OR70}. A representation system
of a represented tensor is not unique, since the map $U$ is
multilinear. Consequently, the matrix $F''(\ve {p}^*)$ is not
positive definite. Therefore, convergence of the linear Gauss-Seidel
method is in general not ensured. However, the convergence of the
ALS method is discussed in \cite{Mohlenkamp2013, UschmajewALS2012,
Wang2014, ZhangGolub2001}. Recently, the convergence of the ALS
method was analysed by means of Lojasiewicz gradient inequality,
please see \cite{Uschmajew:2014} for more details. The current
analysis is not based on the mathematical techniques developed for
the nonlinear Gauss-Seidel method neither on the theory of
Lojasiewicz inequalities, but on the multilinearity of the map $U$.

\begin{note}[$\N_n$]
The set $\N_n$ of natural numbers smaller than $n \in \N$ is denoted
by
\begin{equation*}
    \N_n:=\{j \in \N : 1\leq j \leq n\}.
\end{equation*}
\end{note}

The precise analysis of the ALS method is a quite challenging task.
Some of the difficulties of the theoretical understanding are
explained in the following examples.
\begin{example}\label{exa:mike}
The approximation of $b \in \Vp$ by a tensor of rank one is
considered, where
\begin{eqnarray}\label{eq:defbIntroEx}
  b&=& \sum_{j=1}^r \underbrace{\lambda_j \Ten_{\mu=1}^d b_{j \mu}}_{b_j:=}, \quad \lambda_1 \geq \dots \geq \lambda_r >0, \, \|b_{j
  \mu}\|=1,\\ \nonumber
B_\mu &:=& \left(b_{j \mu} : 1 \leq j \leq r\right) \in \R^{m_\mu
\times r} \quad (1 \leq \mu \leq d),
\end{eqnarray}
and $B_\mu^TB_\mu=\Id$, see the example in \cite[Section
4.3.5]{Mohlenkamp2013}. Let us further assume that  $v_k=p_1^k \ten
p_2^k \ten \dots \ten p_d^k$ is already determined. Corollary
\ref{cor:ALSrecursion} leads to the recursion
\begin{equation}\label{equ:exInt2}
    p_1^{k+1}= \underbrace{\left[\frac{1}{\|v_k\|^2} B_1
    \Diag\left( \lambda_j^2 \prod_{\mu=2}^{d-1} \frac{\dotprod{b_{j
    \mu}}{p_\mu^k}}{\|p_\mu^k\|^2}^2\right)_{j=1, \dots, r}B_1^T\right]}_{G_1(p_1^k, \dots, p_d^k):=} p_1^k \quad (k\geq
    2).
\end{equation}
The linear map $G_1(p_1^k, \dots, p_d^k) \in \R^{m_1 \times m_1}$
describes the first micro step $p_1^k \ten p_2^k \ten \dots \ten
p_d^k \quad \mapsto p_1^{k+1} \ten p_2^k \ten \dots \ten p_d^k$ in
the ALS algorithm. The iteration matrix $G_1(p_1^k, \dots, p_d^k)$
is independent under rescaling of the representation system, i.e.
$G_1(\alpha_1 p_1, \dots, \alpha_d p_d) = G_1(p_1, \dots, p_d)$ for
$1=\prod_{\mu=1}^d \alpha_\mu$. Further, we can illustrate the
difficulties of the ALS iteration in higher dimensions. For $d=2$,
the ALS method is given by the two power iterations
\begin{eqnarray*}
    p_1^{k+1}&=& \left[\frac{1}{\|p_1^k\|^2 \|p_2^k\|^2} B_1
    \Diag\left( \lambda_j^2 \right)_{j=1, \dots, r} B_1^T\right] p_1^k,\\
    p_2^{k+1}&=& \left[\frac{1}{\|p_1^{k+1}\|^2 \|p_2^k\|^2} B_2
    \Diag\left( \lambda_j^2 \right)_{j=1, \dots, r} B_2^T\right] p_2^k.
\end{eqnarray*}
Clearly, if the global minimum $b_1$ is isolated, i.e. $\lambda_1
> \lambda_2$, then the ALS method converges to $b_1$
provided that $\dotprod{v_0}{b_1}\neq0$, where $v_0=p_1^0 \ten p_2^0
\in \Vp$ is the initial guess. Further, we have linear convergence
\begin{equation*}
    \left|\tan\angle[b_{1\,\mu}, p_\mu^{k+1}]\right| \leq \left(\frac{\lambda_2}{\lambda_1}\right)^2\, \left|\tan\angle[b_{1\,\mu}, p_\mu^{k}]\right| \quad (1\leq \mu \leq
    2).
\end{equation*}
Note that in this example the angle $\angle[b_{1\,\mu}, p_\mu^{k}]$
is a more natural measure of the error than the usual distance
$\|b_{1\,\mu} - p_\mu^{k}\|$. For $d\geq 3$, the factor
$\prod_{\mu=2}^{d-1} \dotprod{b_{j \mu}}{p_\mu^k}^2/\|p_\mu^k\|^2$
from Eq. (\ref{equ:exInt2}) describes the behaviour of the ALS
iteration. Let $1 \leq j^*\leq r$. We say that a term $b_{j^*}$ from
Eq. (\ref{eq:defbIntroEx}) dominates at $v_k = p_1^k \ten \dots \ten
p_d^k$ if
\begin{equation}\label{eq:dominates}
    \sqrt[d-2]{\lambda_{j^*}^2}\dotprod{b_{j^*
\mu}}{p_\mu^k}^2 > \sqrt[d-2]{\lambda_{j}^2}\dotprod{b_{j
\mu}}{p_\mu^k}^2
\end{equation}
for all $j \in N_{j^*}:=\left\{j\in \N \, : \, 1\leq j\leq r \mbox{
and } j\neq j^*\right\}$ and all $\mu \in \N_d$. If $b_{j^*}$
dominates at $v_k$, then the recursion formula (\ref{equ:exInt2})
leads to
\begin{equation}\label{eq:superlinAls}
    \left|\tan\angle[b_{j^*\,1}, p_1^{k+1}] \right|\leq \underbrace{\frac{\max_{j \in N_{j^*}} \left( \lambda_{j}\prod_{\mu=2}^{d-1} \dotprod{b_{j
\mu}}{p_\mu^k}\right)^2}{\left(\lambda_{j^*}\prod_{\mu=2}^{d-1}
\dotprod{b_{j^*\mu}}{p_\mu^k}\right)^2}}_{<1}\,\left|\tan\angle[b_{j^*\,1},
p_1^{k}]\right|,
\end{equation}
i.e. the first component of the representation system $p_1^{k+1}$ is
turned towards the direction of $b_{j^*\,1}$. Note that for $r=2$
the bound for the convergence rate is sharp, i.e.
\begin{equation}\label{eq:superlinAlsSharp}
    \left|\tan\angle[b_{j^*\,1}, p_1^{k+1}]\right| = \frac{\max_{j \in N_{j^*}} \left( \lambda_{j}\prod_{\mu=2}^{d-1} \dotprod{b_{j
\mu}}{p_\mu^k}\right)^2}{\left(\lambda_{j^*}\prod_{\mu=2}^{d-1}
\dotprod{b_{j^*\mu}}{p_\mu^k}\right)^2}\,\left|\tan\angle[b_{j^*\,1},
p_1^{k}]\right| \quad (r=2).
\end{equation}
The inequality
\begin{eqnarray*}
    \sqrt[d-2]{\lambda_{j^*}^2}\dotprod{b_{j^*\, 1}}{p_1^{k+1}}^2&=&\frac{1}{\|v_k\|^4}\lambda_{j^*}^4 \prod_{\mu=2}^{d-1}\frac{\dotprod{b_{j^* \mu}}{p_\mu^k}^4}{\|p_\mu^k\|^4}\sqrt[d-2]{\lambda_{j^*}^2}\dotprod{b_{j^*\, 1}}{p_1^{k}}^2\\
    &>&\frac{1}{\|v_k\|^4}\lambda_{j}^4 \prod_{\mu=2}^{d-1}\frac{\dotprod{b_{j \mu}}{p_\mu^k}^4}{\|p_\mu^k\|^4}\sqrt[d-2]{\lambda_{j}^2}\dotprod{b_{j\, 1}}{p_1^{k}}^2 = \sqrt[d-2]{\lambda_{j}^2}\dotprod{b_{j\,
    1}}{p_1^{k+1}}^2
\end{eqnarray*}
shows that $b_{j^*}$ also dominates at the successor $p_1^{k+1} \ten
p_2^k \ten \dots \ten p_d^k$. Further, we have for all $j \in
N_{j^{*}}$
\begin{equation*}
    \frac{\sqrt[d-2]{\lambda_{j}^2}\dotprod{b_{j
,1}}{p_1^{k+1}}^2}{\sqrt[d-2]{\lambda_{j^*}^2}\dotprod{b_{j^*
,1}}{p_1^{k+1}}^2} =
\prod_{\mu=2}^{d-1}\left(\underbrace{\frac{\sqrt[d-2]{\lambda_{j}^2}\dotprod{b_{j
\mu}}{p_\mu^{k}}^2}{\sqrt[d-2]{\lambda_{j^*}^2}\dotprod{b_{j^*
\mu}}{p_\mu^{k}}^2}}_{<1}\right)^2\,\frac{\sqrt[d-2]{\lambda_{j}^2}\dotprod{b_{j
,1}}{p_1^{k}}^2}{\sqrt[d-2]{\lambda_{j^*}^2}\dotprod{b_{j^*
,1}}{p_1^{k}}^2} < \frac{\sqrt[d-2]{\lambda_{j}^2}\dotprod{b_{j
,1}}{p_1^{k}}^2}{\sqrt[d-2]{\lambda_{j^*}^2}\dotprod{b_{j^*
,1}}{p_1^{k}}^2}.
\end{equation*}
By analogy for the following micro steps, we have
\begin{equation*}
\frac{\max_{j \in N_{j^*}} \left( \lambda_{j}\prod_{\mu=2}^{d-1}
\dotprod{b_{j
\mu}}{p_\mu^{k+1}}\right)^2}{\left(\lambda_{j^*}\prod_{\mu=2}^{d-1}
\dotprod{b_{j^*\mu}}{p_\mu^{k+1}}\right)^2} < \frac{\max_{j \in
N_{j^*}} \left( \lambda_{j}\prod_{\mu=2}^{d-1} \dotprod{b_{j
\mu}}{p_\mu^k}\right)^2}{\left(\lambda_{j^*}\prod_{\mu=2}^{d-1}
\dotprod{b_{j^*\mu}}{p_\mu^k}\right)^2}.
\end{equation*}
Hence, the ALS iteration converges to $b_{j*}$. Now it is easy to
see that
\begin{equation*}
    \limsup_{k \rightarrow \infty} \left(\frac{\max_{j \in N_{j^*}} \left( \lambda_{j}  \prod_{\mu=2}^{d-1}
{\dotprod{b_{j
\mu}}{p_\mu^k}}\right)^2}{\left(\lambda_{j^*}\prod_{\mu=2}^{d-1}
{\dotprod{b_{j^*\mu}}{p_\mu^k}}\right)^2}\right)= 0.
\end{equation*}
Therefore, the tangent $\tan\angle[b_{j^*\,\mu}, p_\mu^{k}]$
converges $Q$-superlinearly, i.e.
\begin{equation*}
\left|\tan\angle[b_{j^*\,\mu}, p_\mu^{k}]\right| \xrightarrow[k
\rightarrow \infty]{}0 \quad \quad (\mbox{$Q$-superlinearly}).
\end{equation*}
Furthermore, the ALS iteration converges faster for  large $d$.
Unfortunately, there is no guarantee that the global minimum $b_1$
dominates at $v_k$. However, in this example it is more likely that
a chosen initial guess dominates at the global minimum. For
simplicity let us assume that $r=2$ and $\lambda_1 > \lambda_2$. see
Eq. (\ref{eq:defbIntroEx}). Since the Tucker ranks of $b$ are all
equal to $2$ and the condition from Eq. (\ref{eq:dominates}) does
not depend on the norm of the vectors from the representation
system, assume without loss of generality that for $\mu \in \N_d$
the representation system of every initial guess has the following
form:
\begin{equation*}
    p_\mu(\varphi_\mu)= \sin\left( \varphi_{\mu}\right)b_{\mu,2} + \cos\left(
    \varphi_{\mu}\right)b_{\mu,1}, \quad \left( \varphi_\mu \in \left[0,\frac{\pi}{2}\right], \, \|p_\mu(\varphi_\mu)\|=1\right).
\end{equation*}
If the global minimum dominates at the initial guess, we have for
all $\mu \in \N_d$
\begin{eqnarray*}
    \sqrt[d-2]{\lambda_{1}^2}\dotprod{b_{1
\mu}}{p_\mu(\varphi_\mu)}^2 &>&
\sqrt[d-2]{\lambda_{2}^2}\dotprod{b_{2
\mu}}{p_\mu(\varphi_\mu)}^2\\
   \Leftrightarrow \tan\left(\varphi_\mu\right)&<&
   \sqrt[d-2]{\frac{\lambda_{1}}{\lambda_{2}}}.
\end{eqnarray*}
If we define the angle $\varphi^*_{d,\,\mu} \in \left[0,
\frac{\pi}{2}\right]$ such that

\begin{equation*}
    \tan\left(\varphi^*_{d,
    \,\mu}\right)=\sqrt[d-2]{\frac{\lambda_{1}}{\lambda_{2}}},
\end{equation*}
then every initial guess with $\varphi_\mu \in [0, \varphi^*_{d,
\,\mu})$ converges to the global minimum. Furthermore, we have
\begin{equation*}
  \tan(\varphi^*_{d, \,\mu}) >1  \quad \Leftrightarrow \quad \varphi^*_{d,\,\mu} >
  \frac{\pi}{4},
\end{equation*}
i.e. the slice where the global minimum is a point of attraction is
more potent then the slice where the local minimum $ \lambda_2 b_2$
is a point of attraction, see Figure \ref{bild:globMin} for
illustration. But we have for the asymptotic behavior
\begin{equation*}
\tan\left(\varphi^*_{d,
\,\mu}\right)=\sqrt[d-2]{\frac{\lambda_{1}}{\lambda_{2}}}
\xrightarrow[d \rightarrow \infty]{}1, \quad \Leftrightarrow\quad
\varphi^*_{d, \,\mu} \xrightarrow[d \rightarrow
\infty]{}\frac{\pi}{4},
\end{equation*}
i.e. for sufficiently large $d$ the slices are practically equal
potent.
\begin{figure}[h]
  \centering
  {\image{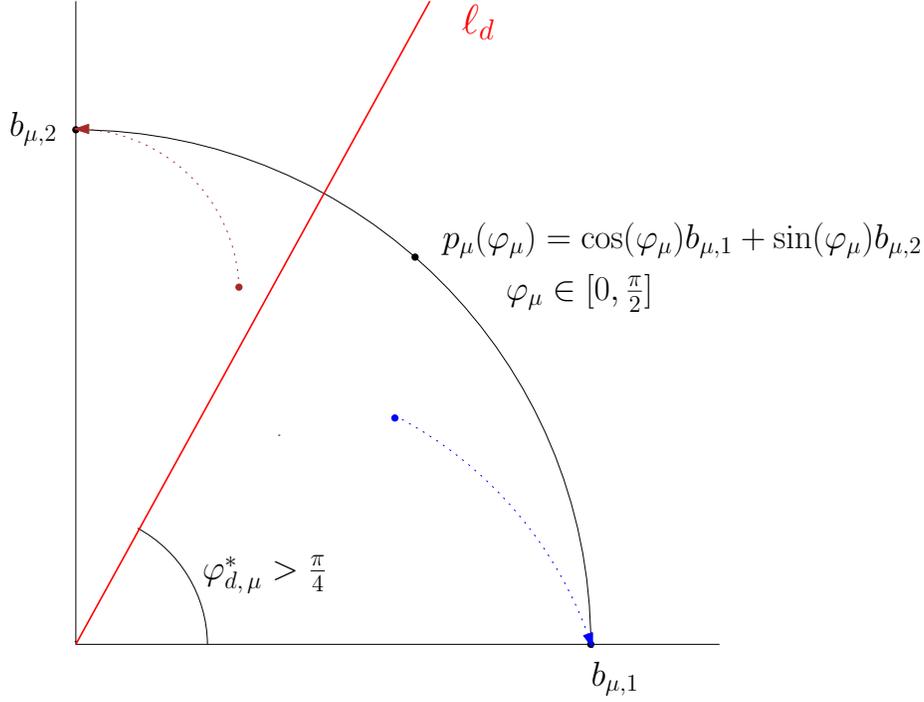}{0.72\textwidth}}
  \caption{The angle $\varphi^*_{d, \,\mu}$ describes the
  slice where the global minimum is a point of attraction. Every initial guess located under the red line $\ell_d$ will converge to the global
  minimum. Note that
  the angle $\varphi^*_{d, \,\mu}$ is larger then $\frac{\pi}{4}$, but interestingly enough $\varphi^*_{d, \,\mu} \xrightarrow[d \rightarrow
\infty]{}\frac{\pi}{4}$.}
  \label{bild:globMin}
\end{figure}

\end{example}

\begin{example}\label{exa:aram}
In the following example a sublinear convergence of ALS procedure for rank-one approximation is shown. We will consider the tensor $b_\lambda \in \Vp$ given by
\begin{equation*}
    b_\lambda = \Ten_{\mu=1}^3 p + \lambda \left(p \ten q \ten q + q \ten p \ten q + q \ten q \ten p \right)
\end{equation*}
for some $\lambda \geq 0$ and $p, q \in \R^n$ with $\|p\| = \|q\| = 1$ and $\dotprod{p}{q} = 0$. Let us first prove the following statement.
\begin{proposition} \label{prop:examplearam}
Define $v^* := \Ten_{\mu=1}^3 p$. Then
\begin{description}
  \item[a)] $\Mp_b = \{ v^* \}$, if $\lambda \leq \frac{1}{2}$
  \item[b)] $| \Mp_b | = 2$ and $v^* \notin \Mp_b$, if $\lambda > \frac{1}{2}$
\end{description}
\end{proposition}
\begin{proof}
Let $v^*_\lambda \in \Mp_b$. Since tensor $b$ is symmetric, $v^*_\lambda$ also has to be symmetric. Write $v^*_\lambda = C_{\lambda} \Ten_{\mu=1}^3 p_\lambda$, where $p_\lambda = p + \alpha_\lambda q$ (this is possible, since $\dotprod{b}{q \Ten q \Ten q} = 0$). Now the tuple $(C_\lambda p_\lambda, p_\lambda, p_\lambda)$ is a stationary point of $F$, therefore
\begin{equation*}
    \left( \Id_{\R^n} \ten p_\lambda \ten p_\lambda \right)^T b = C p_\lambda
\end{equation*}
for some $C \in \R$. But
\begin{equation*}
    \left( \Id_{\R^n} \ten p_\lambda \ten p_\lambda \right)^T b = (1+\lambda \alpha_\lambda^2)p + 2\lambda\alpha_\lambda q,
\end{equation*}
hence
\begin{equation} \label{eq:alphalambda}
    \frac{2\lambda\alpha_\lambda}{1+\lambda \alpha_\lambda^2} = \alpha_\lambda.
\end{equation}
The solutions of \eqref{eq:alphalambda} are
\begin{equation*}
    \alpha_\lambda =  \left\{
    \begin{array}{ll}
    0, & \hbox{if $\lambda \leq \frac{1}{2}$,} \\
    0, \sqrt{\frac{2\lambda-1}{\lambda}} \text{ or } -\sqrt{\frac{2\lambda-1}{\lambda}}, & \hbox{if $\lambda > \frac{1}{2}$.}
    \end{array}
        \right.
\end{equation*}
Straightforward calculations show that for $\lambda > \frac{1}{2}$ the solutions $\alpha_\lambda = \pm \sqrt{\frac{2\lambda-1}{\lambda}}$ lead to the same value of $F$ which is smaller than $f(v^*)$.
\end{proof}
Now let $\lambda \leq \frac{1}{2}$ and $v_k= C^k p_1^k \ten p_2^k \ten \dots \ten p_d^k$, with $p_\mu^k = c_\mu^k p + s_\mu^k q$, ${c_{\mu,k}}^2 + {s_\mu^k}^2 = 1$ and some  $C^k \in\R$. Define $\gamma_{\mu,k} := \left(
                             \begin{array}{c}
                               c_{\mu,k} \\
                               s_{\mu,k} \\
                             \end{array}
                           \right)$. Applying Corollary \ref{cor:ALSrecursion}, one gets after short calculations the recursion formula
\begin{equation*}
    \gamma_{1,k+1} = C_{1,k} M_{1, k} M_{1, k}^T \gamma_{1,k}
\end{equation*}
with some $C_{1,k} \in \R$ and
\begin{equation*}
    M_{1,k} = \left(
              \begin{array}{cc}
                c_{2,k} & \lambda s_{2,k} \\
                \lambda s_{2,k} & \lambda c_{2,k} \\
              \end{array}
            \right).
\end{equation*}
Then for $t_{1,k} := \frac{s_{1,k}}{c_{1,k}}$ it holds
\begin{equation}\label{eq:recursionCS}
     t_{1,k+1} = \frac{\lambda (\lambda+1) c_{2,k} c_{1,k} \frac{s_{2,k}}{s_{1,k}} + \lambda^2}{c^2_{2,k} + \lambda^2 s^2_{2,k} + \lambda (\lambda+1) \frac{c_{2,k}}{c_{1,k}} c_{2,k} s_{1,k}} \frac{s_{1,k}}{c_{1,k}}.
\end{equation}
Thanks to Corollary \ref{cor:ALSisolatedconvergence} and Proposition \ref{prop:examplearam} we know, that $\lim_{k\rightarrow \infty } v^k = v^*$ for $v^* = \Ten_{\mu=1}^3 p$, therefore
\begin{eqnarray}\label{eq:convergenceCS}
    \lim_{k\rightarrow \infty } c_{\mu,k} = 1 \\
    \lim_{k\rightarrow \infty } s_{\mu,k} = 0
\end{eqnarray}
for $\mu \in \N_3$. From Eq. \ref{eq:convergenceCS} and \ref{eq:recursionCS} one gets
\begin{equation}\label{eq:limsup1}
    \limsup_{k\rightarrow \infty} \frac{t_{1,k+1}}{t_{1,k}} = \lambda^2 + \lambda (\lambda+1) \limsup_{k\rightarrow \infty} \frac{s_{2,k}}{s_{1,k}}.
\end{equation}
The same way
\begin{eqnarray}\label{eq:limsup2}
    \limsup_{k\rightarrow \infty} \frac{t_{2,k+1}}{t_{2,k}} = \lambda^2 + \lambda (\lambda+1) \limsup_{k\rightarrow \infty} \frac{s_{3,k}}{s_{2,k}} \\        \limsup_{k\rightarrow \infty} \frac{t_{3, k+1}}{t_{3,k}} = \lambda^2 + \lambda (\lambda+1) \limsup_{k\rightarrow \infty} \frac{s_{1,k+1}}{s_{3,k}}
\end{eqnarray}

Furthermore, from Eq. \eqref{eq:recursion2} we know that
\begin{equation*}
    p_{2,\, k+1} = C_{2,k+1} M_{2,k} p_{1,\,k+1}
\end{equation*}
with some $C_{2,k+1} \in \R$ and
\begin{equation*}
    M_{2,k} = \left(
              \begin{array}{cc}
                c_{3,k} & \lambda s_{3,k} \\
                \lambda s_{3,k} & \lambda c_{3,k} \\
              \end{array}
            \right).
\end{equation*}
Simple calculations result in the relation
\begin{equation*}
    \frac{s_{2,k+1}}{s_{1,k+1}} = \lambda \frac{s_{3,k}}{s_{1,k+1}} c_{1,k+1} + \lambda c_{3,k},
\end{equation*}
and hence
\begin{equation} \label{eq:limsup3}
    \limsup_{k \rightarrow \infty} \frac{s_{2,k}}{s_{1,k}} = \lambda + \lambda \limsup_{k \rightarrow \infty} \frac{s_{3,k}}{s_{1,k+1}}
\end{equation}
Now let $\lambda = \frac{1}{2}$. If $\limsup_{k \rightarrow \infty}
\frac{s_{2,k}}{s_{1,k}} \geq 1$ , then from Eq. \eqref{eq:limsup1}
follows $\limsup_{k\rightarrow \infty} \frac{t_{1,k+1}}{t_{1,k}}
\geq 1$, hence the convergence of $p_{1,k}$ to $p$ can not be
Q-linearly. If $\limsup_{k \rightarrow \infty}
\frac{s_{2,k}}{s_{1,k}} < 1$, then from Eq. \eqref{eq:limsup3}
$\limsup_{k \rightarrow \infty} \frac{s_{1,k+1}}{s_{3,k}} \geq 1$,
so from Eq. \eqref{eq:limsup2} $\limsup_{k\rightarrow \infty}
\frac{t_{3,k+1}}{t_{3,k}} \geq 1$.
\begin{remark}$ $
\begin{itemize}
    \item[a)] In fact for $\lambda = \frac{1}{2}$ it holds
     \begin{equation*}
        \limsup_{k \rightarrow \infty} \frac{s_{2,k}}{s_{1,k}} = \limsup_{k \rightarrow \infty} \frac{s_{3,k}}{s_{2,k}} = \limsup_{k \rightarrow \infty} \frac{s_{1,k+1}}{s_{3,k}} = 1.
     \end{equation*}
    \item[b)] For $\lambda < \frac{1}{2}$ ALS converges q-linearly with the convergence rate
        \begin{equation*}
            \rho = \frac{\lambda}{2} \left(3 \lambda + \lambda^2 + \sqrt{(3 \lambda + \lambda^2)^2 + 4 \lambda}\right).
        \end{equation*}
    \item[c)] The example can be extended to higher dimensions in the following way. Let
    \begin{equation*}
        b_\lambda = \Ten_{\mu=1}^d p + \lambda \sum_{\mu=1}^d \left(\Ten_{\nu=1}^{\mu-1} q \ten p \ten \Ten_{\nu=\mu+1}^d q\right)
    \end{equation*}
    with $\|p\| = \|q\|$ and $\dotprod{p}{q} = 0$. Then $v^* = \Ten_{\mu=1}^d p$ is the unique best rank-one approximation of $b_\lambda$ if and only if $\lambda \leq \frac{1}{d-1}$. Furthermore, ALS converges sublinear for $\lambda = \frac{1}{d-1}$ and Q-linear for $\lambda < \frac{1}{d-1}$.
\end{itemize}
\end{remark}

\end{example}

Our new convergence results are not obtained by using conventional
technics like for the analysis of nonlinear Gauss-Seidel method or
the theory of Lojasiewicz inequalities. Therefore, a detailed
convergence approach is necessary.

\section{The Alternating Least Squares Algorithm}\label{sec:ALSMethod}
In the following section, we recall the ALS algorithm. Where the
algorithmic description of the ALS method is given in Algorithm
\ref{alg:ALS}.

\begin{algorithm}[h]\caption{Alternating Least Squares (ALS)
Algorithm}\label{alg:ALS}
\begin{algorithmic}[1]
    \STATE Set $k:=1$ and choose an initial guess $\ve{p}_1=(p^1_1, \dots, p_d^1)\in
    P$, $\ve{p}_{1,\, 0}:=\ve{p}_1$, and $v_{1}:=U(\ve{p}_1) \neq 0$.
    \WHILE{Stop Condition}
        \STATE $v_{k, \,0}:=v_{k}$
        \FOR {${1 \leq \mu \leq d}$}
        \STATE
            \begin{eqnarray}
                p_{\mu}^{k+1}&:=& \left(\frac{p_1^{k+1}}{\left\|p_1^{k+1}\right\|^2} \ten \cdots \ten \frac{p_{\mu-1}^{k+1}}{\left\|p_{\mu-1}^{k+1}\right\|^2}
                \ten \Id_{\R^{n_\mu}} \ten \frac{p_{\mu+1}^{k}}{\left\|p_{\mu+1}^{k}\right\|^2} \ten \cdots \ten \frac{p_{d}^{k}}{\left\|p_{d}^{k}\right\|^2} \right)^T b\label{eq:defPmu}\\ \nonumber
                \ve{p}_{k, \mu+1}&:=& (p_1^{k+1}, \dots, \,p_{\mu-1}^{k+1} \,, p_\mu^{k+1}, p_{\mu+1}^{k}, \dots,
                p_L^{k})\nonumber\\
                v_{k, \mu+1}&:=&U(\ve{p}_{k, \mu+1})\nonumber
            \end{eqnarray}
        \ENDFOR
    \STATE $\ve{p}_{k+1}:= \ve{p}_{k, L}$ and
    $v_{k+1}:=U(\ve{p}_{k+1})$
    \STATE $k \mapsto k + 1$
    \ENDWHILE
  \end{algorithmic}
\end{algorithm}

\begin{note}[$L(A,B)$, $P_{\nu, \mu}$]
Let $A, B$ be two arbitrary vector spaces. The vector space of
linear maps from $A$ to $B$ is denoted by
\begin{equation*}
 L(A,B):=\left\{ M : A \rightarrow B \, : \, M \mbox{ is linear}\right\}.
\end{equation*}
Let $\mu, \, \nu \in \N_d$ with $\nu\neq\mu$. We define
\begin{equation*}
    P_{\nu, \mu}:=\R^{n_{1}} \times \dots \times \R^{n_{\nu-1}}
\times \R^{n_{\nu+1}} \times \dots \times \R^{n_{\mu-1}} \times
\R^{n_{\mu+1}} \times \dots \times \R^{n_{d}}.
\end{equation*}
\end{note}

The following map $M_{\mu, \nu}$ from Lemma \ref{lemma:ALSrecursion}
is important for the analytical understanding of the ALS algorithm.
As Corollary \ref{cor:ALSrecursion} shows, the map $M_{\mu,
{\mu-1}}$ describes an micro step of the ALS algorithm. Furthermore,
there is an interesting relation between the map $M_{\mu, \nu}$ and
rank-one best approximations of the tensor $b$, see Theorem
\ref{the:singularValues}.

\begin{lemma}\label{lemma:ALSrecursion}
Let $\mu, \nu \in \N_d$, $\nu\neq \mu$, and $\underline{p}_{\nu,\,
\mu}=(p_1,\dots, p_{\nu-1}, p_{\nu+1}, \dots, p_{\mu-1}, p_{\mu+1},
\dots, p_d) \in P_{\nu, \mu}$. There exists a multilinear map
$M_{\nu, \mu} : P_{\nu, \mu} \times \Vp \rightarrow L(\R^{n_{\nu}},
\R^{n_{\mu}})$ such that
\begin{equation}\label{eq:recursion}
   M_{\nu,\mu}(\underline{p}_{\nu,\, \mu}, b) \mathbf{g_\nu}= \left(p_1\ten
\dots \ten p_{\nu-1} \ten \mathbf{g_\nu}\ten p_{\nu+1}\ten \dots\ten
p_{\mu-1}\ten \Id_{\R^{n_\mu}}\ten p_{\mu+1} \ten \dots \ten
p_d\right)^T b
\end{equation}
for all $\mathbf{g_\nu} \in \R^{n_\nu}$. Further, we have
$M_{\mu,\nu}(\underline{p}_{\nu,\, \mu},
b)=M^T_{\nu,\mu}(\underline{p}_{\nu,\, \mu}, b)$.
\end{lemma}
\begin{proof}
Follows directly form the multilinearity of the tensor product and
elementary calculations.
\end{proof}

\begin{example}\label{exa:rank1}
Let $\mu, \nu \in \N_d$, $\nu\neq \mu$, $\underline{p}_{\nu,\,
\mu}=(p_1,\dots, p_{\nu-1}, p_{\nu+1}, \dots, p_{\mu-1}, p_{\mu+1},
\dots, p_d) \in P_{\nu, \mu}$, and $b$ be given in a subspace
decomposition, i.e.
\begin{eqnarray*}
    b&=& \sum_{i_1=1}^{t_1} \dots \sum_{i_d=1}^{t_d} \beta_{(i_1, \dots, i_d)} \Ten_{\mu=1}^d b_{\mu,
    i_\mu} \quad (t_\mu \in \N_{n_\mu})
\end{eqnarray*}
A matrix representation of the linear map $M_{\nu, \, \mu}$ is given
by
\begin{eqnarray*}
M_{\nu,\mu}(\underline{p}_{\nu,\, \mu}, b) &=&
\sum_{i_1=1}^{t_{1}}\cdots \sum_{i_\nu=1}^{t_{\nu}}
\cdots\sum_{i_\mu=1}^{t_{\mu}} \cdots \sum_{i_d=1}^{t_{d}}
\beta_{(i_1, \dots, i_d)} \, \prod_{\xi \in \N_d\setminus\{\mu,\,
\nu\}} \dotprod{b_{\xi, i_\xi}}{p_\xi}
\,b_{\mu, i_\mu} b^T_{\nu, i_\nu}\\
&=&B_\mu \Gamma(\underline{p}_{\nu,\, \mu}) B_{\nu}^T,
\end{eqnarray*}
where $B_{\xi}=\left( b_{\xi, 1}, \dots, b_{\xi, t_\xi}\right) \in
\R^{n_\xi \times t_\xi}$ for all $\xi \in \{\mu, \nu\}$ and the
entries of the matrix $\Gamma(\underline{p}_{\nu,\, \mu})$ are
defined by
\begin{equation*}
[\Gamma(\underline{p}_{\nu,\, \mu})]_{(i_\nu, i_\mu)} =
\sum_{i_1=1}^{t_1}\cdots \sum_{i_{\nu-1}=1}^{t_{\nu-1}}\cdots
\sum_{i_{\nu+1}=1}^{t_{\nu+1}}\cdots\sum_{i_{\mu-1}=1}^{t_{\mu-1}}\cdots\sum_{i_{\mu+1}=1}^{t_{\mu+1}}\cdots
\sum_{i_{d}=1}^{t_{d}}\beta_{(i_1, \dots, i_d)} \, \prod_{\xi \in
\N_d\setminus\{\mu,\, \nu\}} \dotprod{b_{\xi, i_\xi}}{p_\xi}.
\end{equation*}
\end{example}
\begin{corollary}\label{cor:ALSrecursion}
Let $\mu \in \N_d$, $k \geq 2$, and $\ve{p}_{k, \mu}= (p^{k+1}_1,
\dots, p^{k+1}_{\mu-1}, p_\mu^k,p^{k}_{\mu+1}, \dots, p_d^k) \in P$
form Algorithm \ref{alg:ALS}. With the matrix from Lemma
\ref{lemma:ALSrecursion}, the following recursion formula holds:
\begin{eqnarray}\label{eq:recursion1}
p^{k+1}_\mu&=&\frac{1}{G_{k, \, \mu} G_{k,\, \mu-1}} M_{\mu,k}\,
M^T_{\mu,k} \,\, p_{\mu}^{k},
\end{eqnarray}
where
\begin{eqnarray*}
  G_{k, \, \mu} &:=&  \prod_{\nu=1}^{\mu-1} \left\|p_\nu^{k+1}\right\|^2 \prod_{\nu=\mu+1}^{d} \left\|p_\nu^{k}\right\|^2\\
  G_{k, \, \mu-1} &:=&  \prod_{\nu=1}^{\mu-2} \left\|p_\nu^{k+1}\right\|^2 \prod_{\nu=\mu}^{d} \left\|p_\nu^{k}\right\|^2,\\
  M_{\mu,k}&:=& M_{\mu, \mu-1}(p_1^{k+1},
\dots, p_{\mu-2}^{k+1}, p_{\mu+1}^k, \dots, p_d^k, b).
\end{eqnarray*}
\end{corollary}

\begin{proof}
We have with Eq. (\ref{eq:defPmu}) and Lemma
\ref{lemma:ALSrecursion}

\begin{eqnarray}\label{eq:recursion2}
    p^{k+1}_\mu&=&\frac{1}{G_{k, \, \mu}} M_{\mu, \mu-1}(p_1^{k+1},
\dots, p_{\mu-2}^{k+1}, p_{\mu+1}^k, \dots, p_d^k, b)
p_{\mu-1}^{k+1},\\
    p_{\mu-1}^{k+1}&=&\frac{1}{G_{k, \, \mu-1}} M^T_{\mu, \mu-1}(p_1^{k+1},
\dots, p_{\mu-2}^{k+1}, p_{\mu+1}^k, \dots, p_d^k, b) p_{\mu}^{k}.
\end{eqnarray}

\end{proof}
\begin{example}\label{exa:rank1}
Let $v_k = p_1^k \ten p_2^k \ten \dots \ten p_d^k$ and
\begin{equation*}
    b= \sum_{i_1=1}^{t_1} \dots \sum_{i_d=1}^{t_d} \beta_{(i_1, \dots, i_d)} \Ten_{\mu=1}^d b_{\mu,
    i_\mu},
\end{equation*}
i.e. the tensor $b$ is given in the Tucker decomposition. From Eq.
(\ref{eq:defPmu}) it follows
\begin{eqnarray*}
  p_1^{k+1} &=& \frac{1}{\prod_{\mu=2}^d \left\|p_\mu^k\right\|^2}\sum_{i_1=1}^{t_1} \dots \sum_{i_d=1}^{t_d} \beta_{(i_1, \dots,
  i_d)} \prod_{\mu=2}^d \dotprod{b_{\mu, i_\mu}}{p_\mu^k} b_{1,
  i_1}\\
  &=& \frac{1}{\prod_{\mu=2}^{d-1} \left\|p_\mu^k\right\| \|p_d^k\|^2} \left[\sum_{i_1=1}^{t_1} \sum_{i_d=1}^{t_d}  b_{1, i_1} \sum_{i_2=1}^{t_2}\dots \sum_{i_{d-1}=1}^{t_{d-1}} \beta_{(i_1, \dots,
  i_d)} \prod_{\mu=2}^{d-1} \frac{\dotprod{b_{\mu, i_\mu}}{p_\mu^k}}{\|p_\mu^k\|} b^T_{d,
  i_d}\right]p_d^k\\
  &=&\frac{1}{\prod_{\mu=2}^{d-1} \left\|p_\mu^k\right\| \|p_d^k\|^2}B_1
  \Gamma_{1,k} B_d^T p_d^k,
\end{eqnarray*}
where $B_\mu=\left( b_{\mu, i_\mu}\,:\, 1\leq i_\mu \leq
t_\mu\right) \in \R^{n_\mu \times t_\mu}$,
$B_\mu^TB_\mu=\Id_{\R^{t_\mu}}$, and the entries of the matrix
$\Gamma_{1,k} \in \R^{t_1 \times t_d}$ are defined by
\begin{equation*}
    [\Gamma_{1, k}]_{i_1, i_d} = \sum_{i_2=1}^{t_2}
\dots \sum_{i_{d-1}=1}^{t_{d-1}} \beta_{(i_1, \dots,
  i_d)} \prod_{\mu=2}^{d-1} \frac{\dotprod{b_{\mu,
i_\mu}}{p_\mu^k}}{\|p_\mu^k\|} \quad \left( 1\leq i_1 \leq t_1, \, 1
\leq i_d \leq t_d\right).
\end{equation*}
Note that $\Gamma_{1,k}$ is a diagonal matrix if the coefficient
tensor $\beta \in \Ten_{\mu=1}^d \R^{t_\mu}$ is super- diagonal, see
Eq. (\ref{equ:exInt2}).
For $p_d^k$ it follows further
\begin{eqnarray*}
  p_d^{k} &=& \frac{1}{\prod_{\mu=1}^{d-1} \left\|p_\mu^k\right\|^2}\sum_{i_1=1}^{t_1} \dots \sum_{i_d=1}^{t_d} \beta_{(i_1, \dots,
  i_d)} \prod_{\mu=1}^{d-1} \dotprod{b_{\mu, i_\mu}}{p_\mu^k} b_{d,
  i_d}
  =\frac{1}{\left\|p_1^k\right\|^2\prod_{\mu=2}^{d-1} \left\|p_\mu^k\right\|}B_d
  \Gamma^T_{1,k} B_1^T p_1^k
\end{eqnarray*}
and finally
\begin{equation*}
    p_1^{k+1}= \frac{1}{\prod_{\mu=1}^d \left\|p_\mu^k\right\|^2}\, B_1
  \Gamma_{1,k}\Gamma^T_{1,k} B_1^T \,p_1^k.
\end{equation*}
\end{example}

Let $v^*= \lambda \, p_1 \ten \dots \ten p_d \in \Mp_b$ be a
rank-one best approximation of $b$. Without loss of generality we
can assume that
\begin{equation*}
 \|p_1\| = \|p_2\| = \cdots = \|p_d\|=1 \mbox{ and } \|v^*\|=\lambda.
\end{equation*}
Further, let $\mu, \nu \in \N_d$ and
\begin{equation*}
    \underline{p}_{\nu,\,
\mu}:=(p_1,\dots, p_{\nu-1}, p_{\nu+1}, \dots, p_{\mu-1}, p_{\mu+1},
\dots, p_d) \in P_{\nu, \mu}.
\end{equation*}
The following two maps are of interest for our analysis:
\begin{eqnarray*}
  \bar{V} : S^{n_\nu-1} \times S^{n_\mu-1}&\rightarrow& \Vp \\
  (g_\nu, g_\mu) &\mapsto& \bar{V}(g_\nu, g_\mu):=p_1 \ten \cdots
  \ten p_{\nu-1} \ten g_\nu \ten p_{\nu+1} \ten \cdots \ten
  p_{\mu-1} \ten g_\mu \ten p_{\mu+1} \ten \cdots \ten p_d
\end{eqnarray*}
and
\begin{eqnarray*}
  \bar{U} : S^{n_\nu-1} \times S^{n_\mu-1}&\rightarrow& \Vp \\
  (g_\nu, g_\mu) &\mapsto& \bar{U}(g_\nu, g_\mu):=\dotprod{\bar{V}(g_\nu, g_\mu)}{b}\bar{V}(g_\nu,
  g_\mu),
\end{eqnarray*}
where $S^{n-1} = \{x \in \R^n : \|x\| = 1\}$ denotes the sphere in
$\R^n$.

\begin{lemma}\label{lemma:functionValues}
Let $\mu, \nu \in \N_d$, $g_\nu \in S^{n_\nu-1}$ and $g_\mu \in
S^{n_\mu-1}$. We have
\begin{eqnarray*}
  -2f\left(\bar{U}(g_\nu, g_\mu)\right) &=& \dotprod{\underbrace{\left(M_{\nu,\mu}(\underline{p}_{\nu,\, \mu},
  b)\right)}_{\in L(\R^{n_\nu}, \R^{n_\mu})}
  g_{\nu}}{g_\mu}^2 = \dotprod{\bar{U}(g_\nu, g_\mu)}{b} = \left\|\bar{U}(g_\nu,
  g_\mu)\right\|^2.
\end{eqnarray*}
\end{lemma}
\begin{proof}
Let $g_\nu \in S^{n_\nu-1}$, $g_\mu \in S^{n_\mu-1}$, and define $
\pi(g_\nu, g_\mu):=\bar{V}(g_\nu, g_\mu)(\bar{V}(g_\nu, g_\mu))^T$.
It holds $\bar{U}(g_\nu, g_\mu)= \pi(g_\nu, g_\mu)b$ and
\begin{eqnarray*}
  f\left(\bar{U}(g_\nu, g_\mu)\right) &=& \frac{1}{2} \dotprod{\bar{U}(g_\nu, g_\mu)}{\bar{U}(g_\nu,
  g_\mu)}-\dotprod{\bar{U}(g_\nu, g_\mu)}{b}= \frac{1}{2} \dotprod{\pi^2(g_\nu, g_\mu)b}{b}-\dotprod{\pi(g_\nu, g_\mu)b}{b}\\
  &=& \frac{1}{2} \dotprod{\pi(g_\nu, g_\mu)b}{b}-\dotprod{\pi(g_\nu,
  g_\mu)b}{b}=- \frac{1}{2} \dotprod{\pi(g_\nu, g_\mu)b}{b}=- \frac{1}{2}\dotprod{\bar{U}(g_\nu, g_\mu)}{b}\\
  &=&- \frac{1}{2} \dotprod{\pi^2(g_\nu, g_\mu)b}{b} = - \frac{1}{2} \dotprod{\pi(g_\nu, g_\mu)b}{\pi(g_\nu,
  g_\mu)b}=- \frac{1}{2} \left\|\bar{U}(g_\nu,
  g_\mu)\right\|^2\\&=&- \frac{1}{2} \dotprod{\bar{V}(g_\nu,
  g_\mu)}{b}^2=- \frac{1}{2}\dotprod{\left(M_{\nu,\mu}(\underline{p}_{\nu,\, \mu},
  b)\right)
  g_{\nu}}{g_\mu}^2.
\end{eqnarray*}
\end{proof}
\begin{remark}\label{rem:LagMin}
Obviously, the minimisation problem from Eg.
(\ref{eq:defMinProblem}) is equivalent to the following constrained
maximisation problem: Find $\tilde{v}=\Ten_{\mu=1}^d p_\mu$ such
that for all $\mu \in \N_d$ it holds
\begin{equation*}
    \dotprod{\tilde{v}}{b} = \max_{v\in U(P)} \dotprod{v}{b} \quad
    \mbox{subject to } \|p_\mu\|=1.
\end{equation*}
Lagrangian method for constrained optimisation leads to
\begin{equation*}
    L_{\ve{\lambda}}(q_1, \cdots, q_d) = \dotprod{U(q_1, \cdots, q_d)}{b} +
    \frac{1}{2}\sum_{\mu=1}^d \lambda_\mu\left(1-\|q_\mu\|^2\right),
\end{equation*}
where  $q_\mu \in \R^{n_\mu}$ and $\ve{\lambda}=(\lambda_1, \cdots,
\lambda_d)^T \in \R^d$ is the vector of Lagrange multipliers. A
rank-one best approximation $v^*= \lambda \, p_1 \ten \cdots \ten
p_d \in \Mp_b$  with $\lambda\in \R$ and $\|p_\mu\|=1$ satisfies
\begin{eqnarray*}
  \frac{\partial }{\partial p_\mu} L_{\ve{\lambda}^*}(p_1, \cdots, p_d)&=& \left(p_1 \ten \cdots p_{\mu-1 } \ten \Id_{\R^{n_\mu}} \ten p_{\mu+1} \ten
  \cdots \ten p_d\right)^T b  - \lambda^*_\mu p_\mu =0,\\
  \frac{\partial }{\partial \lambda^*_\mu}L_{\ve{\lambda}^*}(p_1, \cdots, p_d)
  &=& \frac{1}{2}\left(1 - \|p_\mu\|^2\right)=0.
\end{eqnarray*}
For $\nu \in \N_d \setminus \{\mu\}$ it follows that
\begin{eqnarray*}
   \lambda &=& \dotprod{p_1 \ten \cdots \ten p_d}{b}, \quad \lambda \, p_\mu =  M_{\nu,\mu}(\underline{p}_{\nu,\, \mu}, b) \,
   p_\nu, \quad \lambda \, p_\nu =  M^T_{\nu,\mu}(\underline{p}_{\nu,\, \mu}, b) \, p_\mu,
\end{eqnarray*}
where $\ve{p}_{\nu, \mu} \in P_{\nu, \mu}$ is like in Lemma
\ref{lemma:ALSrecursion}. Therefore, $\lambda $ is a singular value
of the matrix $M_{\nu,\mu}(\underline{p}_{\nu,\, \mu}, b)$ and
$p_\nu,\, p_\mu$ are the associated singular vectors.
\end{remark}
\begin{proposition}\label{prop:globalMin}
Let $v^* = \lambda p_1 \ten \cdots \ten p_d \in \Mp_b$ a best
approximation of $b$ with $\|p_1\|= \cdots = \|p_d\|=1$. We have
\begin{equation*}
    f(v^*) = -\frac{1}{2 \, \|b\|^2} \|v^*\|^2 = -\frac{1}{2 \, \|b\|^2}
    \dotprod{b}{v^*}.
\end{equation*}
\end{proposition}
\begin{proof}
Since $v^* \in \Mp_b$ we have that $v^*=\Pi b$, where $\Pi :=
\frac{v^*{v^*}^T}{\|v^*\|^2}$. Furthermore, it holds
\begin{eqnarray*}
  \dotprod{v^*}{v^*} &=& \dotprod{\Pi b}{v^*} = \dotprod{b}{\Pi v^*}
  =\dotprod{b}{v^*}.
\end{eqnarray*}
The rest follows from the definition of $f$, see Eq.
(\ref{equ:deff}).
\end{proof}
\begin{remark}
From Proposition \ref{prop:globalMin} it follows instantly that the
global minimum of the best approximation problem from Eq.
(\ref{eq:defMinProblem}) has the largest norm among all other
$\tilde{v} \in \Mp_b$.
\end{remark}

\begin{theorem}\label{the:singularValues}
Let $\mu, \nu \in \N_d$ and $v^*=\|v^*\| p_1 \ten  \dots \ten p_d
\in \Mp_b$ be a rank-one best approximation of $b$ with $\|p_1\|=
\cdots =\|p_d\|=1$. Then $\|v^*\|$ is the largest singular value of
$M_{\nu,\mu}(\underline{p}_{\nu,\, \mu}, b)$ and $p_\nu,\, p_\mu$
are the associated singular vectors. Furthermore, if $v^*$ is
isolated, then $\|v^*\|$ is a simple singular value of
$M_{\nu,\mu}(\underline{p}_{\nu,\, \mu}, b)$.
\end{theorem}
\begin{proof}
Let $\mu, \, \nu \in \N_d$. From Lemma \ref{lemma:functionValues}
and Remark \ref{rem:LagMin} it follows that $\|v^*\|$ is a singular
value of $M_{\nu,\mu}(\underline{p}_{\nu,\, \mu}, b)$ and $p_\nu,\,
p_\mu$ are associated singular vectors. Assume that there is a
singular value $\tilde{\lambda}$ of
$M_{\nu,\mu}(\underline{p}_{\nu,\, \mu}, b)$ and associated singular
vectors $q_\nu \in \R^{n_\nu}, q_\mu \in \R^{n_\mu}$ with
$\tilde{\lambda}>\|v^*\|$. Let $\alpha \in [0,1]$ and $\beta \in
(0,1]$ with $\alpha^2+\beta^2=1$. Define further
$g_{\nu}(\alpha,\beta):=g_{\nu}:= \alpha p_\nu + \beta q_\nu \in
\R^{n_\nu}$ and $g_{\mu}(\alpha,\beta):=g_{\mu}:= \alpha p_\mu +
\beta q_\mu \in \R^{n_\mu}$. We have $\|g_{\nu}\|^2=\|g_{\mu}\|^2=
\alpha^2 + \beta^2=1$ and with Lemma \ref{lemma:functionValues} it
follows then
\begin{eqnarray*}
  -2f(\bar{U}(g_\nu, g_\mu)) &=& \dotprod{\left(M_{\nu,\mu}(\underline{p}_{\nu,\, \mu},
  b)\right)
  g_{\nu}}{g_\mu}^2 = \dotprod{\left(M_{\nu,\mu}(\underline{p}_{\nu,\, \mu},
  b)\right)\alpha p_\nu + \beta q_\nu }{\alpha p_\mu + \beta q_\mu
  }^2\\
  &=& \dotprod{\alpha \|v^*\|p_\nu + \beta \tilde{\lambda} q_\nu }{\alpha p_\mu + \beta q_\mu
  }^2= \left( \alpha^2 \|v^*\| + \beta^2 \tilde{\lambda}\right)^2\\
  &\overset{(\beta\neq 0)}{>} &\left( \alpha^2 \|v^*\| + \beta^2
  \|v^*\|\right)^2 = \|v^*\|^2 = \dotprod{\left(M_{\nu,\mu}(\underline{p}_{\nu,\, \mu},
  b)\right)p_{\nu}}{p_\mu}^2 = -2f(v^*).
\end{eqnarray*}
Consequently, it is
\begin{equation*}
    f(\bar{U}(g_{\nu}(\alpha,\beta), g_{\mu}(\alpha,\beta)))<f(v^*) \quad \mbox{for all } \alpha \in
    [0,1] \mbox{ and } \beta \in (0,1] \mbox{ with } \alpha^2 +
    \beta^2 =1,
\end{equation*}
i.e. we can finde a better approximation
$\bar{U}(g_{\nu}(\alpha,\beta), g_{\mu}(\alpha,\beta))$
of $b$ which is arbitrary close to $v^*$. This contradicts the fact that $v^* \in \Mp_b$.\\
Additionally, let $v^*$ be a isolated rank-one best approximation of
$b$. Assume that there is a singular value $\lambda$ of
$M_{\nu,\mu}(\underline{p}_{\nu,\, \mu}, b)$ and associated singular
vectors $q_\nu \in \R^{n_\nu}, \, q_\mu \in \R^{m_\mu}$ with
$\lambda=\|v^*\|$, $p_\nu \bot q_\nu$, and $p_\mu \bot q_\mu$.
Almost like above, let $\alpha,\,\beta \in [0,1]$ with $\alpha^2 +
\beta^2=1$ and consider again $g_{\nu}(\alpha,\beta)= \alpha p_\nu +
\beta q_\nu \in \R^{n_\nu}$, $g_{\mu}(\alpha,\beta)= \alpha p_\mu +
\beta q_\mu \in \R^{n_\mu}$. With Lemma \ref{lemma:functionValues}
it follows
\begin{eqnarray*}
  -2f(\bar{U}(g_\nu, g_\mu)) &=& \left( \alpha^2 \|v^*\| + \beta^2 \lambda\right)^2 = \|v^*\|^2 = \dotprod{\left(M_{\nu,\mu}(\underline{p}_{\nu,\, \mu},
  b)\right)p_{\nu}}{p_\mu}^2 = -2f(v^*),
\end{eqnarray*}
i.e. we have
\begin{equation*}
    f(\bar{U}(g_{\nu}(\alpha,\beta), g_{\mu}(\alpha,\beta)))=f(v^*) \quad \mbox{for all } \alpha, \, \beta \in
    [0,1] \mbox{ with } \alpha^2 +
    \beta^2 =1.
\end{equation*}
Therefore, we can finde a approximation
$\bar{U}(g_{\nu}(\alpha,\beta), g_{\mu}(\alpha,\beta))$ of $b$ which
is arbitrary close to $v^*$ and $f(\bar{U}(g_{\nu}(\alpha,\beta),
g_{\mu}(\alpha,\beta)))=f(v^*)$. This contradicts the fact that
$v^*$ is isolated.
\end{proof}

\begin{remark}
The proof of Theorem \ref{the:singularValues} shows that if we have
two different best approximations of $b$ which differ only in two
arbitrary components of the representation systems and
$f(v^*)=f(v^{**})$, then there is a complete path between $v^*$ and
$v^{**}$ described by $\bar{U}(g_\nu(\alpha,\beta),
g_\mu(\alpha,\beta))$ such that $f(v^*) =
f\left(\bar{U}(g_\nu(\alpha,\beta), g_\mu(\alpha,\beta))\right)$.
\end{remark}

\section{Convergence Analysis}\label{sec:Analyse}
In the following, we are using the notations and definitions from
Section \ref{sec:ALSMethod}. Our convergence analysis is mainly
based on the recursion introduced in Corollary
\ref{cor:ALSrecursion} and the following Lemma
\ref{lemma:Projection}.

\begin{lemma}\label{lemma:Projection}
Let $k \in \N$, $\mu \in \N$, and $v_{k,\mu}=p_1^{k+1} \ten \cdots
\ten p_{\mu-1}^{k+1} \ten p_{\mu}^{k} \ten \cdots \ten p_d^{k}$ from
Algorithm \ref{alg:ALS}. Then
\begin{equation*}
    \Pi_{k,\mu}:=
    \frac{p_1^{k+1}\left(p_1^{k+1}\right)^T}{\left\|p_1^{k+1}\right\|^2} \ten
    \cdots \ten
    \frac{p_{\mu-1}^{k+1}\left(p_{\mu-1}^{k+1}\right)^T}{\left\|p_{\mu-1}^{k+1}\right\|^2}
    \ten \Id_{\R^{n_\mu}} \ten
    \frac{p_{\mu+1}^{k}\left(p_{\mu+1}^{k}\right)^T}{\left\|p_{\mu+1}^{k}\right\|^2}
    \ten \cdots \ten \frac{p_d^{k+1}\left(p_d^{k+1}\right)^T}{\left\|p_d^{k+1}\right\|^2}
\end{equation*}
is a orthogonal projection and
\begin{equation*}
    v_{k, \mu+1} = v_{k, \mu} + \Pi_{k,\mu} r_{k, \mu},
\end{equation*}
where $r_{k,\mu}:= b - v_{k,\mu}$.
\end{lemma}
\begin{proof}
Obviously, $\Pi_{k,\mu}$ is a orthogonal projection. Straightforward
calculations show that $v_{k, \mu}=\Pi_{k,\mu}v_{k, \mu}$ and $v_{k,
\mu+1}=\Pi_{k,\mu}b$. Hence we have $v_{k, \mu} + \Pi_{k,\mu} r_{k,
\mu} = \Pi_{k,\mu}b = v_{k, \mu+1}$.
\end{proof}
\begin{lemma}\label{lemma:fmicro}
Let $k\in \N$, $\mu \in \N_L$. We have
\begin{equation}\label{eq:fmicro}
    f(v_{k, \, \mu}) - f(v_{k, \mu + 1}) = \frac{1}{2} \frac{\dotprod{\Pi_{k, \mu} r_{k,
    \mu}}{r_{k, \mu }}}{\|b\|^2}
\end{equation}
\end{lemma}
\begin{proof}
It follows with Lemma \ref{lemma:Projection} that
\begin{eqnarray*}
    f(v_{k, \mu + 1})&=&\frac{1}{\|b\|^2}\left[ \frac{1}{2} \dotprod{v_{k, \mu} + \Pi_{k,\mu} r_{k, \mu}}{v_{k, \mu} + \Pi_{k,\mu} r_{k, \mu}} - \dotprod{b}{v_{k, \mu} + \Pi_{k,\mu} r_{k,
    \mu}}\right]\\
&=& f(v_{k, \mu})+\frac{1}{\|b\|^2}\left[\frac{1}{2}
\dotprod{\Pi_{k, \mu} r_{k, \mu}} {\Pi_{k, \mu} r_{k, \mu}} +
\dotprod{ v_{k, \mu}} {\Pi_{k, \mu} r_{k, \mu}} - \dotprod{b}
{\Pi_{k, \mu} r_{k, \mu}}  \right]\\
&=&f(v_{k, \mu})+\frac{1}{\|b\|^2}\left[\frac{1}{2} \dotprod{r_{k,
\mu}} {\Pi_{k, \mu} r_{k, \mu}} - \dotprod{ r_{k, \mu}} {\Pi_{k,
\mu} r_{k, \mu}}  \right]\\
&=& f(v_{k, \mu}) - \frac{1}{2}\frac{\dotprod{\Pi_{k, \mu} r_{k,
    \mu}}{r_{k, \mu }}}{\|b\|^2},
\end{eqnarray*}
i.e. $f(v_{k, \, \mu}) - f(v_{k, \mu + 1}) = \frac{1}{2}
\frac{\dotprod{\Pi_{k, \mu} r_{k,
    \mu}}{r_{k, \mu }}}{\|b\|^2}$.
\end{proof}
\begin{corollary}\label{cor:functionValues}
There exists $\alpha \in \R$ such that $f(v_k)  \xrightarrow[k
\rightarrow \infty]{} \alpha$.
\end{corollary}
\begin{proof}
Let $k \in \N$ and $\mu \in \N_L$. From Lemma \ref{lemma:fmicro} and
Lemma \ref{lemma:Projection} it follows that
\begin{eqnarray*}
  f(v_{k+1}) - f(v_k) &=& f(v_{k, d}) - f(v_{k, 0}) =
  \sum_{\mu=1}^d f(v_{k, \, \mu}) - f(v_{k, \mu -1})\\ &=&- \frac{1}{2 \|b\|^2}
  \sum_{\mu=0}^{d-1} \left\| \Pi_{k, \mu}r_{k,
    \mu}\right\|^2 \leq 0,
\end{eqnarray*}
This shows that $(f(v_k))_{k\in \N} \subset \R$ is a descending
sequence. The sequence of function values $(f(v_k))_{k\in \N}$ is
bounded from below. Therefore, there exist an $\alpha \in \R$ such
that $f(v_k) \xrightarrow[k \rightarrow \infty]{} \alpha$.
\end{proof}
\begin{remark}
From the definition of the ALS method it is already clear that
$(f(v_{k, \mu}))_{\mu \in \N_d, k\in \N}$ is a descending sequence.
\end{remark}
\begin{lemma}\label{lemma:functionValues2}
Let $(v_{k,\mu})_{k\in \N, \mu\in\N_d} \subset \Vp$ be the sequence
from Algorithm \ref{alg:ALS}. We have
\begin{equation}\label{eq:lemmafunctionValues}
f(v_{k,\mu}) = -\frac{1}{2\|b\|^2}\dotprod{
v_{k,\mu}}{b}=-\frac{1}{2\|b\|^2}\|v_{k,\mu}\|^2
\end{equation}
for all $k\in \N, \mu\in\N_d$.
\end{lemma}
\begin{proof}
Let $k\in \N$ and $\mu\in\N_d$. With Lemma \ref{lemma:Projection} it
follows
\begin{eqnarray*}
  \dotprod{v_{k, \mu}}{v_{k, \mu}} &=& \dotprod{\Pi_{k, \mu-1}b}{\Pi_{k, \mu-1}b} = \dotprod{\Pi^2_{k, \mu-1}b}{b} = \dotprod{\Pi_{k, \mu-1}b}{b}= \dotprod{v_{k, \mu}}{b}.
\end{eqnarray*}
The rest follows from the definition of $f$, see Eq.
(\ref{equ:deff}).
\end{proof}
\begin{corollary}\label{cor:equf} Let $(v_{k,\mu})_{k\in \N, \mu\in\N_d} \subset \Vp$ be the sequence of represented
tensors from the ALS algorithm. Further, let $\mu \in \N_d$ and $k
\in \N$. The following statements are equivalent:
\begin{itemize}
  \item [(a)] $f(v_{k, \mu+1}) \leq f(v_{k, \mu})$
  \item [(b)] $\|v_{k, \mu+1}\|^2 \geq \|v_{k, \mu}\|^2$
  \item [(c)] $\|p^{k+1}_{\mu}\|^2 \geq \|p^{k}_{\mu}\|^2$
  \item [(d)] $\cos^2 (\phi_{k, \mu+1}) \geq \cos^2 (\phi_{k, \mu})$,
  where $\cos^2 (\phi_{k, \mu}):= \frac{\dotprod{\Pi_{k, \mu}b}{b}}{\|b\|^2}$.
\end{itemize}
\end{corollary}
\begin{proof}
Follows direct from Lemma \ref{lemma:functionValues2} and
\begin{equation*}
    \|v_{k, \mu+1}\|^2 \geq \|v_{k, \mu}\|^2 \Leftrightarrow G_{k, \mu}\|p^{k+1}_{\mu}\|^2 \geq G_{k, \mu} \|p^{k}_{\mu}\|^2
\end{equation*}
where $G_{k, \mu}>0$ is defined in Corollary \ref{cor:ALSrecursion}.
\end{proof}
\begin{lemma}\label{lemma:dist}
Let $(v_k)_{k\in \N} \subset \Vp$ be the sequence of represented
tensors from the ALS method. It holds
\begin{equation*}
    \|v_{k+1} - v_{k}\|\xrightarrow[k \rightarrow \infty]{}0.
\end{equation*}
\end{lemma}

\begin{proof}
 Let $k \in \N$. We have
\begin{eqnarray}\label{eq:profvkNull}
  \|v_{k+1} - v_{k}\|^2 &=& \left\| \sum_{\mu=1}^d v_{k,\mu} - v_{k,
  \mu-1}\right\|^2 \leq \left( \sum_{\mu=1}^d \|v_{k,\mu} - v_{k,
  \mu-1}\| \right)^2 \leq d \sum_{\mu=0}^{d-1} \left\|v_{k,\mu+1} - v_{k,
  \mu}\right\|^2.
\end{eqnarray}
Since $v_{k, \mu+1} - v_{k, \mu} = \Pi_{k,\mu} r_{k,\mu}$, see Lemma
\ref{lemma:Projection}, it follows further with Eq.
(\ref{eq:fmicro}) and (\ref{eq:profvkNull}) that
\begin{eqnarray*}
  \|v_{k+1} - v_{k}\|^2 &\leq& 2 d \|b\|^2 \sum_{\mu=0}^{d-1} \left(f(v_{k, \mu+1}) - f(v_{k, \mu})\right).
\end{eqnarray*}

With Corollary \ref{cor:functionValues} we have $\left(f(v_{k,
\mu+1}) - f(v_{k, \mu})\right)\xrightarrow[k \rightarrow \infty]{}
0$, hence $\|v_{k+1} - v_{k}\|\xrightarrow[k \rightarrow
\infty]{}0$.
\end{proof}
\begin{defn}[$\Ap(v_k)$, critical points]
Let $(v_k)_{k \in \N} \subset \Vp$ be the sequence of represented
tensors from Algorithm \ref{alg:ALS}. The set of accumulation points
of $(v_k)_{k \in \N}$ is denoted by $\Ap(v_k)$, i.e.
\begin{equation}\label{eq:AccumulationPoint}
    \Ap(v_k):=\left\{ v \in \Vp : v \mbox{ is an accumulation point of } (v_k)_{k \in \N} \right\}.
\end{equation}
The set $\mathfrak{M}$ of \emph{critical points} of the optimisation
problem from Eq. (\ref{equ:defF}) is defined as follows:
\begin{equation}\label{eq:criticalPoint}
    \mathfrak{M}:=\left\{ v \in \Vp \,: \, \exists \ve{p} \in P : v=U(\ve{p}) \wedge F'(\ve{p})=0\right\}.
\end{equation}
\end{defn}
\begin{proposition}\label{proPos:beschrParameter}
The sequence of parameter $(p_{\mu, k})_{\mu \in \N_d, k\in \N}$
from the ALS algorithm is bounded.
\end{proposition}
\begin{proof}
From the definition of $f$ and Lemma \ref{lemma:functionValues2} it
follows that
\begin{equation*}
    -\frac{1}{2} \leq f(v_{k, \mu}) = -\frac{1}{2} \frac{\|v_{k,
    \mu}\|^2}{\|b\|^2} \quad \Leftrightarrow \quad\|v_{k, \mu}\| \leq
    \|b\|,
\end{equation*}
i.e. the sequence $(\|v_{\mu, k}\|)_{\mu \in \N_d, k\in \N} \subset
\Bild(U)$ is bounded. The sequence $(\|v_{\mu, k}\|)_{\mu \in \N_d,
k\in \N}$ is the product of the following $d$ sequences
$(\|p_\mu^k\|)_{k \in \N} \subset \R^{n_\mu}$. According to
Corollary \ref{cor:equf} the sequences $(\|p_\mu^k\|)_{k \in \N}$
are monotonically increasing. Since the product $\|v_{\mu, k}\|$ is
bounded and all sequences $(\|p_\mu^k\|)_{k \in \N}$ are
monotonically increasing, it follows that all $(p_\mu^k)_{k \in \N}$
are bounded. This means the sequence $(p_{\mu, k})_{\mu \in \N_d,
k\in \N}$ is bounded.
\end{proof}
The following statements are proofed in a corresponding article
about the convergence of alternating least squares optimisation in
general tensor format representations, please see \cite{ESHAKH13_1}
for more informations regarding the proofs.
\begin{lemma}[\cite{ESHAKH13_1}]\label{lemma:gradNull}
We have
\begin{equation*}
    \max_{0\leq \mu\leq L-1}\left\|F_\mu'(p_\mu^k)\right\| \xrightarrow[k \rightarrow \infty]{}0.
\end{equation*}
\end{lemma}
\begin{corollary}[\cite{ESHAKH13_1}]\label{cor:gradF2Null}
Let $(\ve{p}_k)_{k\in \N}$ be the sequence from Algorithm
\ref{alg:ALS} and $F: P \rightarrow \R$ from Eq. (\ref{equ:defF}).
We have
\begin{equation*}
    \lim_{k\rightarrow \infty } F'(\ve{p}_k)=0.
\end{equation*}
\end{corollary}

\begin{theorem}[\cite{ESHAKH13_1}]\label{theo:AcuCrit}
Let $(v_k)_{k \in \N}$ be the sequence of represented tensors from
the ALS method. Every accumulation point of $(v_k)_{k \in \N}$ is a
critical point, i.e. $\Ap(v_k) \subseteq \mathfrak{M}$. Further, we
have
\begin{equation*}
    \dist{v_k}  {\mathfrak{M}} \xrightarrow[k \rightarrow \infty]{}0.
\end{equation*}
\end{theorem}
Let $\bar{v} \in \mathfrak{M}$ be a critical point and
$N:=\prod_{\mu=1}^d n_\mu \in \N$. Further, let $(\ve{p}_{k,
\mu})_{k\in\N, \mu \in \N_d} \subset P$ be the sequence of parameter
from the ALS algorithm and $R \in \R^{N \times N-1}$ be a matrix
with $R^T R = \Id_{\R^{N-1}}$ and $\Span(\bar{v})^\bot = \Bild(R)$,
i.e. the column vectors of $R$ build an orthonormal basis of the
linear space $\Span(\bar{v})^\bot$. Then the block matrix
\begin{equation}\label{equ:defOrthBasis}
    V:=\left[
         \begin{array}{cc}
         \ve{v} & R\\
         \end{array}
       \right]\in \R^{N \times N}, \quad  \left(\,\ve{v}:=\bar{v}/\|\bar{v}\| \,\right).
\end{equation}
is orthogonal, i.e. the columns of the matrix $V$ build an
orthonormal basis of the tensor space $\Vp$. The following matrix
$N_{k, \mu} \in \R^{N \times N}$ is imported in order to describe
the rate of convergence for the ALS method:
\begin{equation*}
    N_{k, \mu} := \Ten_{\nu=1}^{\mu-1} \Id \ten \left(\frac{1}{G_{k, \, \mu} G_{k,\, \mu-1}} M_{\mu,k}\,
M^T_{\mu,k} \right)\ten \Ten_{\nu=\mu+1}^d \Id,
\end{equation*}
where the matrix $\frac{1}{G_{k, \, \mu} G_{k,\, \mu-1}} M_{\mu,k}\,
M^T_{\mu,k}$ is from Corollary \ref{cor:ALSrecursion}. Further, it
follows from Corollary \ref{cor:ALSrecursion} that for the ALS micro
step the following equation:
\begin{equation}\label{equ:Recursion}
    v_{k, \mu+1} = N_{k,\mu} v_{k, \mu}
\end{equation}
holds. The tensor $v_{k,\mu}$ and the matrix $N_{k,\mu}$ are
represented with respect to the basis $V$, i.e
\begin{eqnarray*}
  v_{k, \mu} &=& V V^Tv_{k, \mu}= \left[
         \begin{array}{cc}
         \ve{v} & R\\
         \end{array}
       \right]\left(
                \begin{array}{c}
                  \underbrace{\ve{v}^Tv_{k,\mu}}_{c_{k,\mu}:=} \\
                  \underbrace{R^T v_{k,\mu}}_{s_{k,\mu}:=} \\
                \end{array}
              \right)= \left[
         \begin{array}{cc}
         \ve{v}& R\\
         \end{array}
       \right]\left(
                \begin{array}{c}
                  c_{k,\mu} \\
                  s_{k,\mu} \\
                \end{array}
              \right)
\end{eqnarray*}
and
\begin{eqnarray*}
  N_{k, \mu} &=& V \left(V^T N_{k, \mu} V \right)V^T = \left[
         \begin{array}{cc}
         \ve{v} & R\\
         \end{array}
       \right] \left[
       \begin{array}{cc}
       \ve{v}^T N_{k, \mu} \ve{v} &  \ve{v}^T N_{k, \mu} R\\
       R^T N_{k, \mu} \ve{v}&  R^T N_{k, \mu} R\\
       \end{array}\right]
  \left[
         \begin{array}{cc}
         \ve{v} & R\\
         \end{array}
       \right]^T.
\end{eqnarray*}
The recursion formula (\ref{equ:Recursion}) leads to the recursion
of the coefficient vector
\begin{eqnarray*}
\left( \begin{array}{c}
        c_{k+1,\mu} \\
        s_{k+1,\mu} \\
        \end{array}
        \right) = \left[
       \begin{array}{cc}
       \ve{v}^T N_{k, \mu} \ve{v} &  \ve{v}^T N_{k, \mu} R\\
       R^T N_{k, \mu} \ve{v}&  R^T N_{k, \mu} R\\
       \end{array}\right]\left(
                \begin{array}{c}
                  c_{k,\mu} \\
                  s_{k,\mu} \\
                \end{array}
              \right)=\left(
                \begin{array}{c}
                  \ve{v}^T N_{k, \mu} \ve{v} \,\, c_{k,\mu} + \ve{v}^T N_{k, \mu} R \,\, s_{k,\mu}\\
                  R^T N_{k, \mu} \ve{v} \,\, c_{k,\mu} +  R^T N_{k, \mu} R \,\, s_{k,\mu} \\
                \end{array}
              \right).
\end{eqnarray*}
Without loss of generality we can assume that $\|s_{k, \mu}\| \neq
0$ and $|c_{k, \mu}| \neq 0$. Therefore, the following terms are
well defined:
\begin{eqnarray*}
  q_{k,\mu}^{(s)} &:=&\frac{\left\|R^T N_{k, \mu} \ve{v} \,\, c_{k,\mu} +  R^T N_{k, \mu} R \,\, s_{k,\mu} \right\|}{\|s_{k,\mu}\|},  \\
  q_{k,\mu}^{(c)} &:=&\frac{\left| \ve{v}^T N_{k, \mu} \ve{v} \,\, c_{k,\mu} + \ve{v}^T N_{k, \mu} R \,\,
  s_{k,\mu}\right|}{|c_{k,\mu}|}.
\end{eqnarray*}
This preconsideration gives a recursion formula for the tangent of
the angle between $\bar{v}$ and $v_{k, \mu+1}$. We have
\begin{eqnarray*}
  \tan^2\angle[\bar{v}, v_{k, \mu+1}]  &=& \frac{\dotprod{RR^T v_{k, \mu+1}}{v_{k, \mu+1}}}{\dotprod{\ve{v}\ve{v}^T v_{k, \mu+1}}{v_{k,
  \mu+1}}} = \frac{\|R^T v_{k, \mu+1}\|^2}{\left(\ve{v}^T v_{k,
  \mu+1}\right)^2}=\frac{\|s_{k, \mu+1}\|^2}{(c_{k, \mu+1})^2}=\frac{\left(q_{k,\mu}^{(s)}\right)^2}{\left(q_{k,\mu}^{(c)}\right)^2}\frac{\|s_{k, \mu}\|^2}{\left(c_{k,
  \mu}\right)^2}\\
  &=&\left(\frac{q_{k,\mu}^{(s)}}{q_{k,\mu}^{(c)}}\right)^2
  \frac{\|R^T v_{k, \mu}\|^2}{\left(\ve{v}^T v_{k,
  \mu}\right)^2} = \left(\frac{q_{k,\mu}^{(s)}}{q_{k,\mu}^{(c)}}\right)^2 \tan^2\angle[\bar{v}, v_{k, \mu}].
\end{eqnarray*}

\begin{remark}\label{rem:APnotEmpty}
Obviously, if the sequence of parameter $(\ve{p}_k)_{k \in
\N}\subset P$ is bounded, then the set of accumulation points of
$(\ve{p}_k)_{k \in \N}$ is not empty. Consequently, the set
$\Ap(v_k)$ is not empty, since the map $U$ is continuous.
\end{remark}
\begin{theorem}[\cite{ESHAKH13_1}]\label{theorem:isoAP}
If one accumulation point $\bar{v} \in \Ap(v_k) \subseteq
\mathfrak{M}$ is isolated, then we have
\begin{equation*}
    v_k \xrightarrow[k \rightarrow \infty]{}\bar{v}.
\end{equation*}
Furthermore, we have for the rate of convergence of an ALS micro
step
\begin{equation*}
    \left|\tan\angle[\bar{v}, v_{k, \mu+1}]\right| \leq q_\mu \left|\tan\angle[\bar{v}, v_{k, \mu}]
    \right|,
\end{equation*}
where
\begin{equation*}
    q_\mu:= \limsup_{k \rightarrow
    \infty}\left| \frac{q_{k,\mu}^{(s)}}{q_{k,\mu}^{(c)}}\right|.
\end{equation*}
If $q_\mu=0$, then the sequence $\left( \left|\tan\angle[\bar{v},
v_{k, \mu}]\right|\right)_{k \in \N}$ converges Q- superlinearly. If
$q_\mu <1$, then the sequence $\left( \left|\tan\angle[\bar{v},
v_{k, \mu}]\right|\right)_{k \in \N}$ converges at least Q-
linearly. If $q_\mu\geq 1$, then the sequence $\left(
\left|\tan\angle[\bar{v}, v_{k, \mu}]\right|\right)_{k \in \N}$
converges not Q-linearly.
\end{theorem}

\begin{remark}
The calculation from Example \ref{exa:mike} shows that
\begin{equation*}
    \limsup_{k \rightarrow
    \infty}\left| \frac{q_{k,\mu}^{(s)}}{q_{k,\mu}^{(c)}}\right|=0 \quad \mbox{for all } \mu \in \N_d.
\end{equation*}
Hence, the ALS algorithm converges here Q-superlinearly.
Furthermore, in Example \ref{exa:aram} we showed for $\lambda <
\frac{1}{2}$
\begin{equation*}
\limsup_{k \rightarrow
    \infty}\left| \frac{q_{k,\mu}^{(s)}}{q_{k,\mu}^{(c)}}\right|=\frac{\lambda}{2} \left(3 \lambda + \lambda^2 + \sqrt{(3 \lambda + \lambda^2)^2 + 4 \lambda}\right) < 1 \quad \mbox{for all } \mu \in \N_d.
\end{equation*}
Hence, we have here Q-linear convergence.
\end{remark}

\begin{corollary}[\cite{ESHAKH13_1}]\label{cor:ALSisolatedconvergence}
If the set of critical points $\mathfrak{M}$ is
discrete,\footnote{In topology, a set which is made up only of
isolated points is called discrete.} then the sequence of
represented tensors $(v_k)_{k \in \N}$ from the ALS method is
convergent.
\end{corollary}
In the following example it will be shown, that the ordering of the indices may play an important role for the convergence of ALS procedure.
\begin{remark}
Let $b = \Ten_{\mu=1}^3 b_{1\mu} + \lambda \Ten_{\mu=1}^3 b_{2\mu}$,
with $0 < \lambda < 1$, $\|b_{1\mu}\| = \|b_{2\mu}\| = 1$ and
$\dotprod{b_{1\mu}}{b_{2\mu}}=0$ for $\mu \in \N_{\leq d}$. Let
further $v^0 = C \Ten_{\mu=1}^d p_1^0$ for some $C \in \R$ and
\begin{equation}\label{eq:exampleform}
    p_\mu^0 = b_{1\mu} + \alpha_\mu b_{2\mu}
\end{equation}
for some $\alpha_\mu \in \R$. Assume after each ALS micro step the
parameters $p_\mu^k$ are rescaled to the form (\ref{eq:exampleform})
(obviously, a scaling of parameters has no effect on the future
behavior of the ALS method). After the first four micro steps one
gets
\begin{eqnarray*}
    p^1_1 = b_{11} + \lambda \alpha_2 \alpha_3 b_{21} \\
    p^1_2 = b_{12} + \lambda^2 \alpha_2 \alpha_3^2 b_{22} \\
    p^1_3 = b_{13} + \lambda^4 \alpha_2^2 \alpha_3^3 b_{23} \\
    p^2_1 = b_{11} + \lambda^7 \alpha_2^3 \alpha_3^5 b_{21}
\end{eqnarray*}
So for $v_1^2 := p^2_1 \ten p^1_2 \ten p^1_3$ one gets
\begin{equation*}
     v_1^2 = \hat{C} (b_{11} + \lambda^7 \alpha_2^3 \alpha_3^5 b_{21}) \ten (b_{13} + \lambda^2 \alpha_2 \alpha_3^2 b_{23}) \ten (b_{12} + \lambda^4 \alpha_2^2 \alpha_3^3 b_{22})
\end{equation*}
with some $\hat{C} \in \R$. Now assume the order of the directions
for ALS optimization is changed from $(1, 2, 3)$ to $(1, 3, 2)$,
i.e. after optimizing the first component $p^1_1$ we optimize the
third one (i.e. $p^1_3$) and only then the second one (i.e.
$p^1_2$). The same number of micro steps will result in a tensor
\begin{equation*}
    \tilde{v_1^2} = \tilde{C} (b_{11} + \lambda^7 \alpha_2^5 \alpha_3^2 b_{21}) \ten (b_{13} + \lambda^4 \alpha_2^3 \alpha_3^2 b_{23}) \ten (b_{12} + \lambda^2 \alpha_2^2 \alpha_3 b_{22})
\end{equation*}
with some $\tilde{C} \in \R$. Now if $\alpha_2$ and $\alpha_3$ satisfy
\begin{eqnarray*}
    \alpha_2 \geq 1 \geq \alpha_3, \\
    \alpha_2^3 \alpha_3^2 \geq \frac{1}{\lambda^5} \geq \alpha_2^2 \alpha_3^3,
\end{eqnarray*}
then it is not difficult to check, that $v_1^2$ satisfies the
dominance condition from Eq. (\ref{eq:dominates}) for $j = 1$,
whereas $\tilde{v_1^2}$ satisfies the dominance condition for $j =
2$. Thus, with the same starting point $v^0$ ALS iteration will
converge to the global minimum $\Ten_{\mu=1}^d b_{1\mu}$ for one
ordering of the indices and to local minimum $\lambda \Ten_{\mu=1}^d
b_{2\mu}$ for another ordering. Note that $v_0$ did not fulfil the
dominance conditions, but depending on the ordering of the ALS micro
steps $v_0$ leads to a dominance condition for different terms.
\end{remark}

\section{Numerical Experiments}
In this subsection, we observe the convergence behavior of the ALS
method by using data from interesting examples and more importantly
from real applications. In all cases, we focus particularly on the
convergence rate.

\subsection{Example 1}
We consider an example introduced by Mohlenkamp in \cite[Section
4.3.5]{Mohlenkamp2013}. Here we have
\begin{equation*}
    b = 2 \underbrace{\underbrace{\left(
            \begin{array}{c}
              1 \\
              0 \\
            \end{array}
          \right)}_{e_1:=} \ten \left(
            \begin{array}{c}
              1 \\
              0 \\
            \end{array}
          \right) \ten \left(
            \begin{array}{c}
              1 \\
              0 \\
            \end{array}
          \right)}_{b_1:=} + \underbrace{\underbrace{\left(
            \begin{array}{c}
              0 \\
              1 \\
            \end{array}
          \right)}_{e_2:=} \ten \left(
            \begin{array}{c}
              0 \\
              1 \\
            \end{array}
          \right) \ten \left(
            \begin{array}{c}
              0 \\
              1 \\
            \end{array}
          \right)}_{b_2:=},
\end{equation*}
see Eq. (\ref{equ:deff}). The tensor $b$ is orthogonally
decomposable. Although the example is rather simple, it is of
theoretical interest. Since the ALS method converges superlinear,
cf. the discussion in Section \ref{sec:introduction}. The tensor $b$
has only two terms, therefore the upper bound for convergence rate
from Eq. (\ref{eq:superlinAls}) is sharp, cf. Eq.
(\ref{eq:superlinAlsSharp}). Let $\tau\geq 0$, we define the initial
guess of the ALS algorithm by
\begin{equation*}
    v_0(\tau):= \left(
               \begin{array}{c}
                 \tau \\
                 1 \\
               \end{array}
             \right) \ten \left(
                            \begin{array}{c}
                              \tau \\
                              1 \\
                            \end{array}
                          \right) \ten \left(
                            \begin{array}{c}
                              \tau \\
                              1 \\
                            \end{array}
                          \right).
\end{equation*}
Since
\begin{equation*}
    4\dotprod{\left(\begin{array}{c}
                 1 \\
                 0 \\
               \end{array}
             \right)}{\left(\begin{array}{c}
                  \tau \\
                 1 \\
               \end{array}
             \right)}^2=4 \tau^2 \quad \mbox{and} \quad     \dotprod{\left(\begin{array}{c}
                 0 \\
                 1 \\
               \end{array}
             \right)}{\left(\begin{array}{c}
                 \tau \\
                 1 \\
               \end{array}
             \right)}^2=1,
\end{equation*}
we have for $\tau <\frac{1}{2}$ that the initial guess $v_0(\tau)$
dominates at $b_2$. Therefore, the ALS iteration converge to $b_2$.
If $\tau > \frac{1}{2}$, then  $v_0(\tau)$ dominates at $b_1$ and
the sequence from the ALS method will converges to $b_1$. In the
first test the tangents of the angle between the current iteration
point and the corresponding parameter of the dominate term $b_l$
($1\leq l \leq 2$) is plotted, i.e.
\begin{eqnarray}\label{eq:defTan}
  \tan \phi_{k, l}
  &=&\sqrt{\frac{1-\cos^2{\phi_{k,l}}}{\cos^2{\phi_{k,l}}}},
\end{eqnarray}
where $\cos{\phi_{k,l}} = \frac{\dotprod{p_1^k}{e_l}}{\|p_1^k\|}$.
To illustrate the superlinear convergence of the ALS method, we
present further plots for the quotient
\begin{equation}\label{eq:defAlpha}
    q_{k, l}:=\frac{\tan{\phi_{k+1,l}}}{\tan{\phi_{k,l}}}.
\end{equation}
\begin{figure}[h]
  \centering
  {\begin{turn}{0}  \image{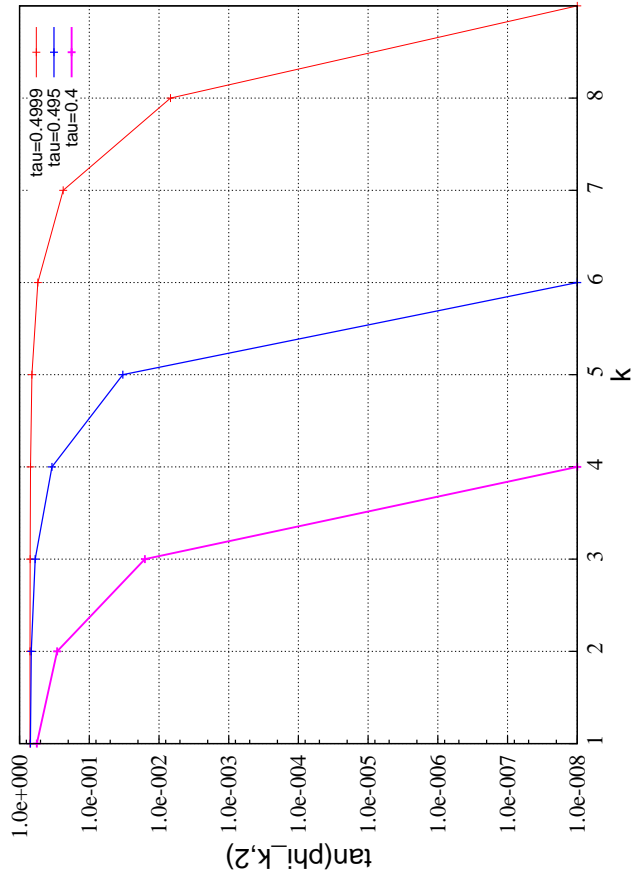}{0.5\textwidth} \end{turn}}
  \caption{The tangents $\tan \phi_{k, 2}$ from Eq. (\ref{eq:defTan}) is plotted for $\tau \in \{0.4, \, 0.495, \,0.4999\}$.}
  \label{bild:tan}
\end{figure}
\begin{figure}[h]
  \centering
  \subfigure[$q_{k, 1}$ is plotted for $\tau\in \{0.5001, \,0.505, \, 0.6\}$. Here the term $b_1$ dominates at every iteration point.]{{\begin{turn}{0}  \image{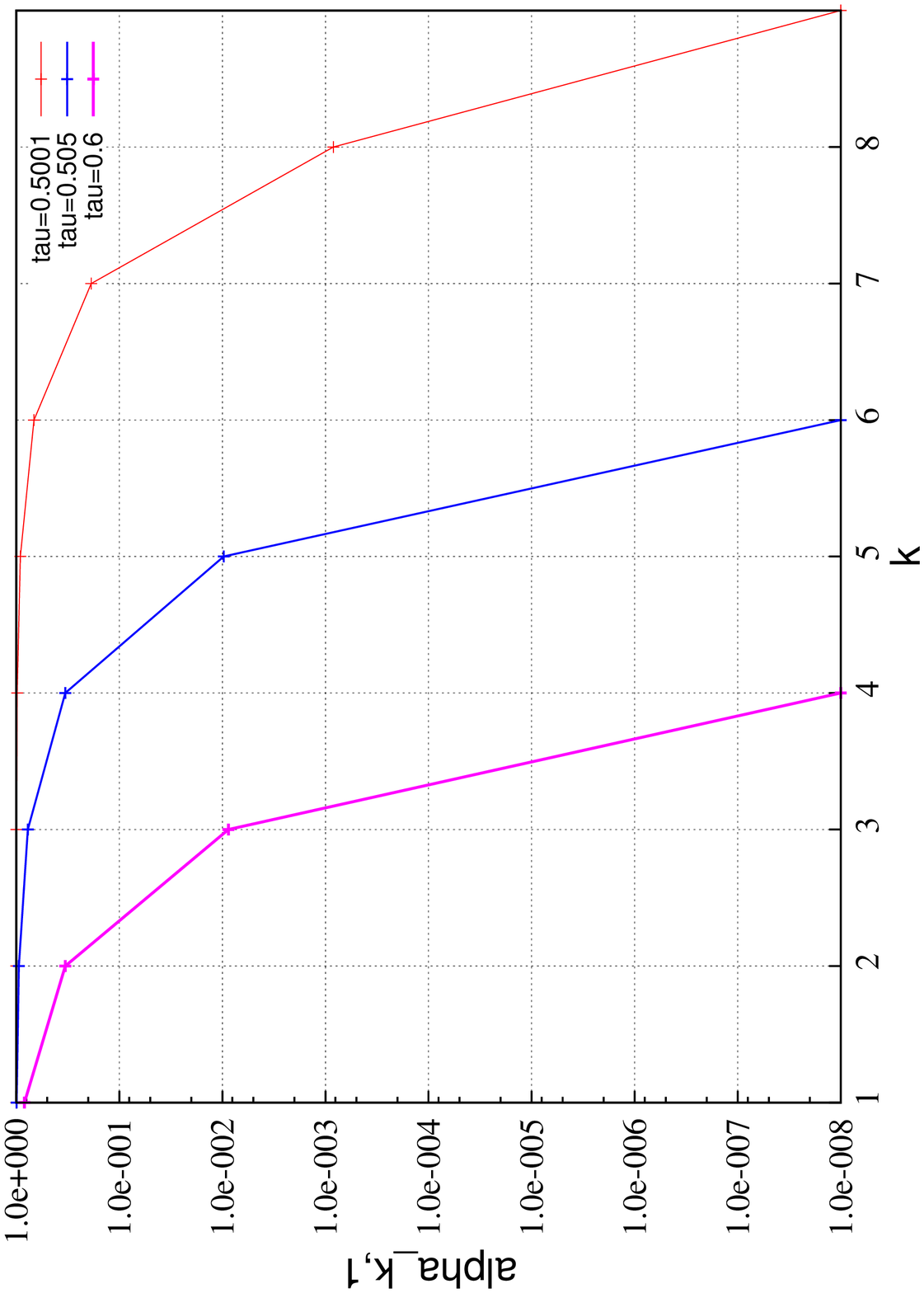}{0.45\textwidth} \end{turn}}}\label{bild:b1}\hfill
  \subfigure[$q_{k, 2}$ is plotted for $\tau\in \{0.4999, \,0.495, \, 0.4\}$. Here the term $b_2$ dominates at every iteration point.]{{\begin{turn}{0}  \image{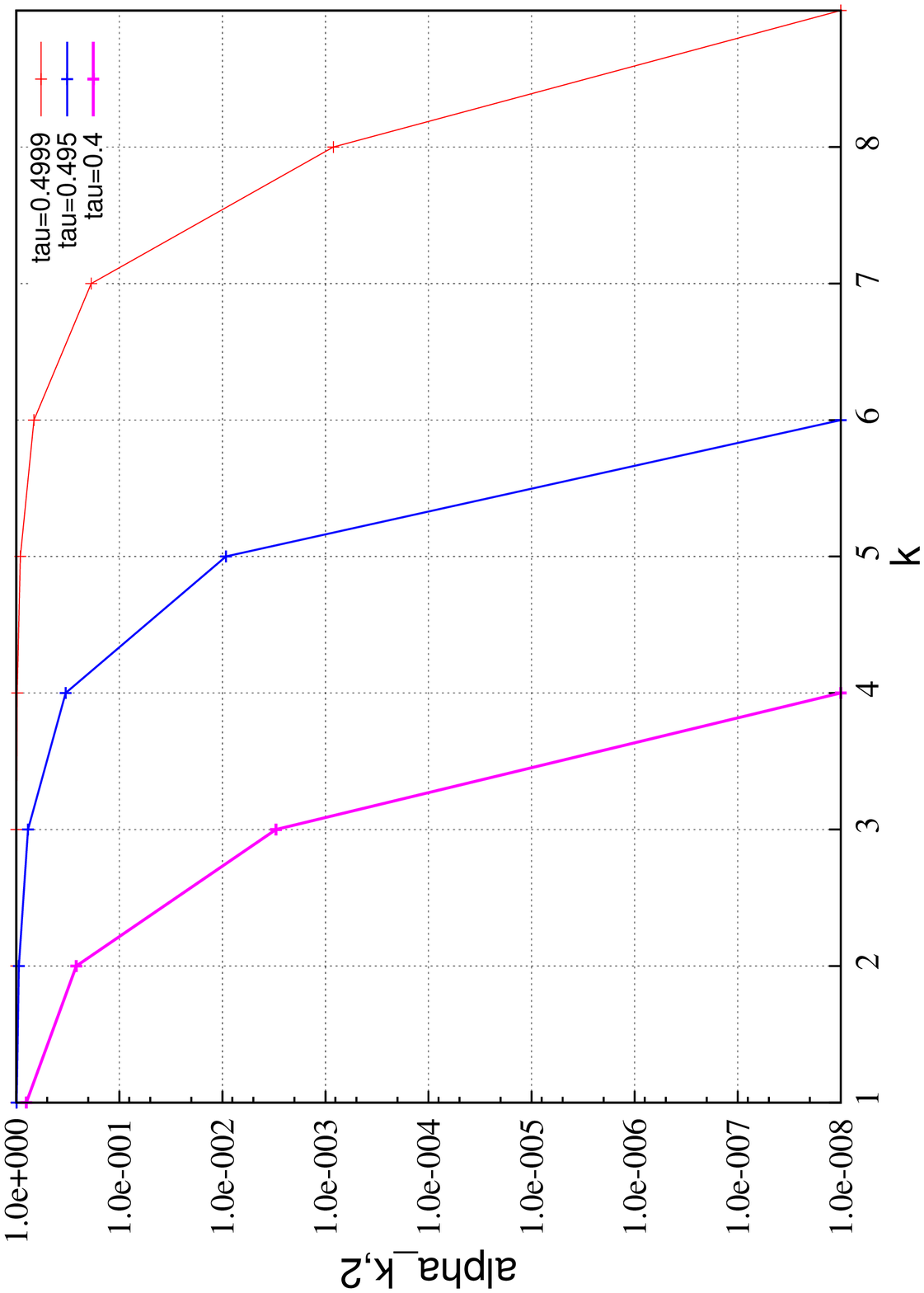}{0.45\textwidth} \end{turn}}}\label{bild:b2}
  \caption{$q_{k, l}$ from Eq. (\ref{eq:defAlpha}) is plotted for $l \in \{1, 2\}$ and different values for $\tau$.}
\end{figure}

\subsection{Example 2}
Most algorithms in ab initio electronic structure theory compute
quantities in terms of one- and two-electron integrals. In
\cite{esbe2011} we considered the low-rank approximation of the
two-electron integrals. In order to demonstrate the convergence of
the ALS method on an example of practical interest, we use the order
$4$ tensor for the two-electron integrals of the so called AO basis
for the CH$_4$ molecule. We refer the reader to \cite{esbe2011} for
a detailed description our example. In this example the ALS method
converges Q-linearly, see Figure \ref{bild:CH4}.

\begin{figure}[h]
  \centering
  {\begin{turn}{0}  \image{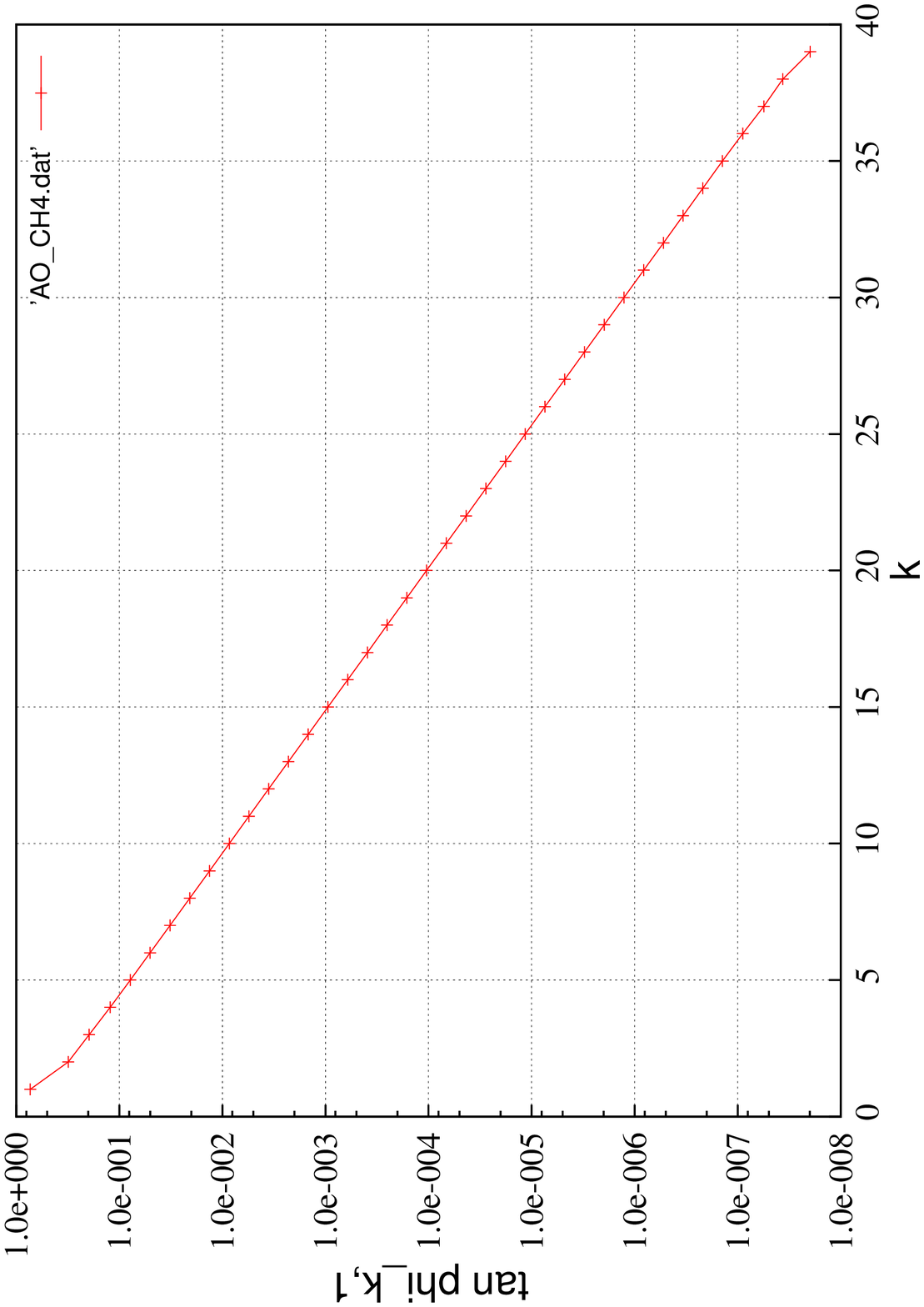}{0.5\textwidth} \end{turn}}
  \caption{The approximation of two-electron integrals for methane is considered. The tangents of the angle between the current iteration point and the
limit point with respect to the iteration number is plotted.}
  \label{bild:CH4}
\end{figure}

\subsection{Example 3}
We consider the tensor
\begin{equation*}
    b_\lambda = \Ten_{\mu=1}^3 p + \lambda \left(p \ten q \ten q + q \ten p \ten q + q \ten q \ten p \right)
\end{equation*}
from Ex. \ref{exa:aram}. The vectors $p$ and $q$ are arbitrarily
generated orthogonal vectors with norm $1$. The values of $\tan
(\phi_k^1)$ are plotted, where $\phi_k^1$ is the angle between
$p_k^1$ and the limit point $p$ (i.e. $\tan \phi_k^1 =
\frac{\dotprod{p_k^1}{q}}{\dotprod{p_k^1}{p}}$, for $k \geq 2$). For
the case $\lambda = 0.5$ the convergence is sublinearly, whereas for
$\lambda = 0.2$ it is Q-linearly.

\begin{figure}[h]
  \centering
  \subfigure[The tangents $\tan \phi_{k, 1}$ for $\lambda = 0.2$.]{{ \begin{turn}{0}  \image{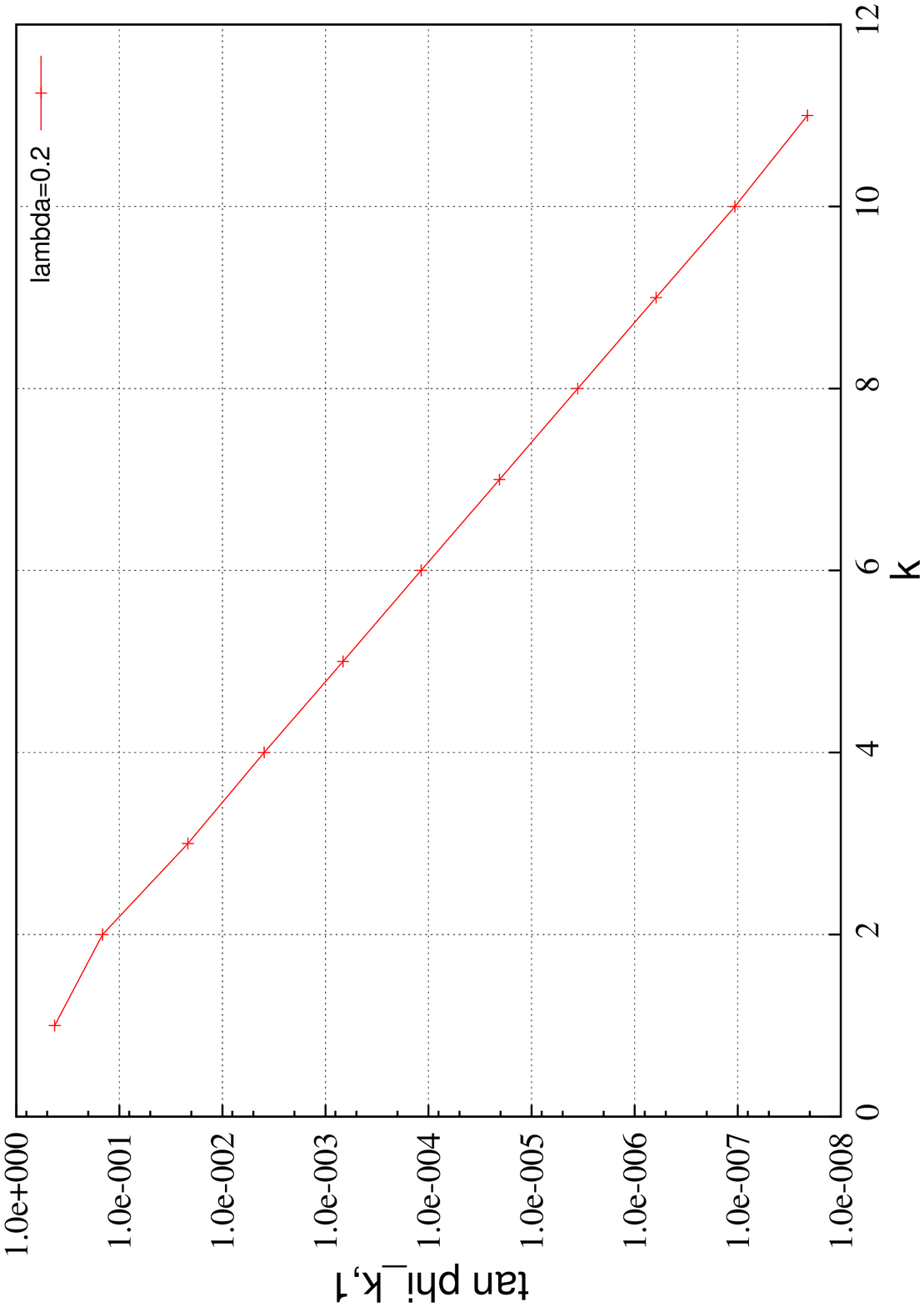}{0.45\textwidth} \end{turn}}}\label{bild:data02}\hfill
  \subfigure[The tangents $\tan \phi_{k, 1}$ for $\lambda = 0.5$.]{{ \begin{turn}{0}\image{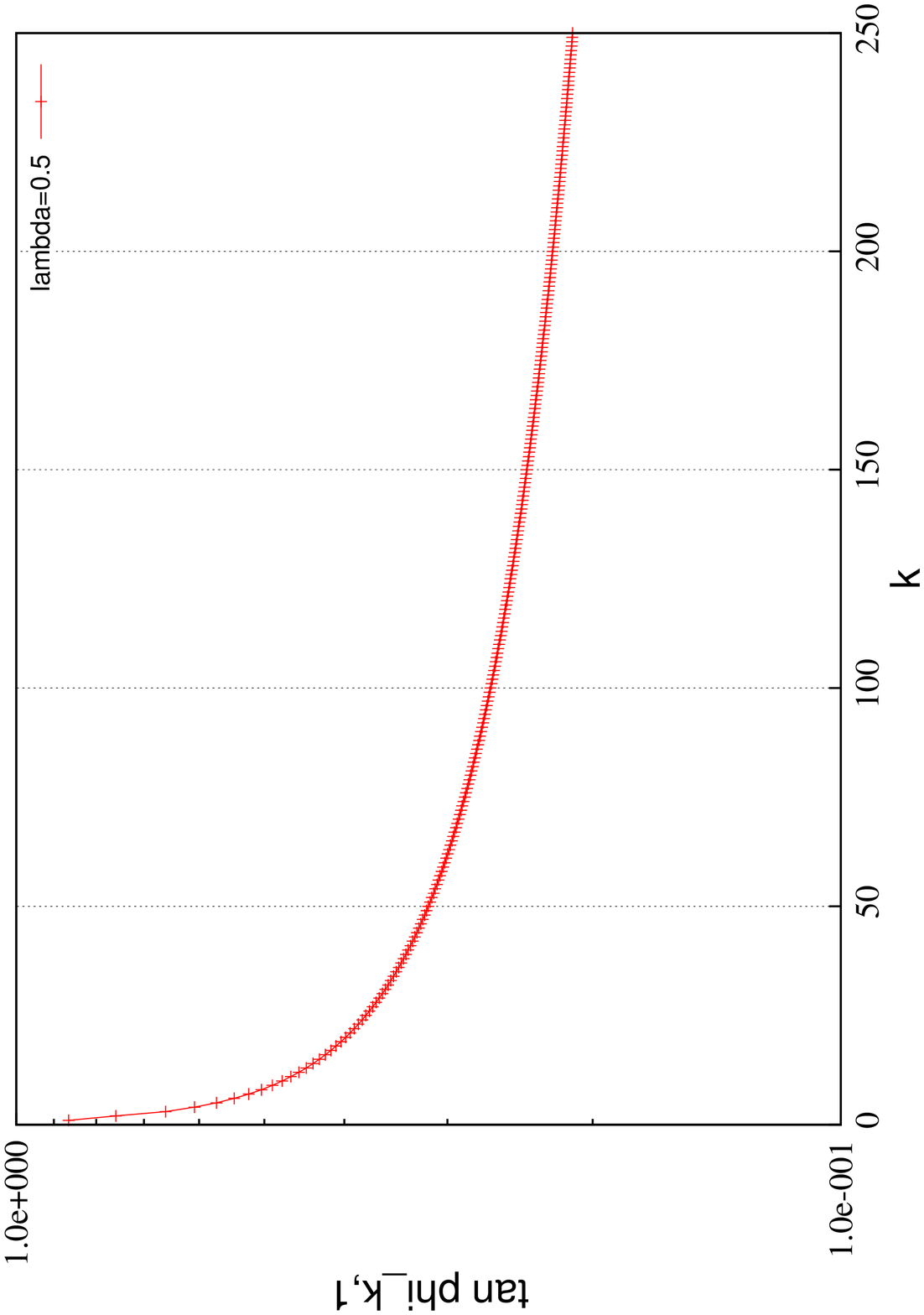}{0.45\textwidth}   \end{turn}}}\label{bild:data05}
  \caption{The approximation of $b$ from Example \ref{exa:aram} is considered. The tangents of the angle between the current iteration point and the
limit point with respect to the iteration number is plotted. For
$\lambda=1/2$, we have sublinear convergence. But for $\lambda = 0.2
< 1/2$ the sequence converges Q-linearly.}
\end{figure}
%


%

\bibliographystyle{plain}
\bibliography{all}

\begin{thebibliography}{10}

\bibitem{esbe2011}
U.~Benedikt, A.~Auer, M.~Espig, and W.~Hackbusch.
\newblock Tensor decomposition in post-hartree–fock methods. i. two-electron
  integrals and mp2.
\newblock {\em The Journal of Chemical Physics}, 134(5):--, 2011.

\bibitem{BEMO02}
G.~Beylkin and M.~J. Mohlenkamp.
\newblock Numerical operator calculus in higher dimensions.
\newblock {\em Proceedings of the National Academy of Sciences},
  99(16):10246--10251, 2002.

\bibitem{BEMO05}
G.~Beylkin and M.~J. Mohlenkamp.
\newblock Algorithms for numerical analysis in high dimensions.
\newblock {\em SIAM Journal on Scientific Computing}, 26(6):2133--2159, 2005.

\bibitem{ESHAHARS11_2}
M.~Espig, W.~Hackbusch, S.~Handschuh, and R.~Schneider.
\newblock Optimization problems in contracted tensor networks.
\newblock {\em Computing and Visualization in Science}, 14(6):271--285, 2011.

\bibitem{ESHAKH13_1}
M.~Espig, W.~Hackbusch, and A.~Khachatryan.
\newblock On the convergence of alternating least squares optimisation in
  tensor format representations.
\newblock {\em Preprint}, 2014.

\bibitem{HA12}
W.~Hackbusch.
\newblock {\em Tensor Spaces and Numerical Tensor Calculus}.
\newblock Springer, 2012.

\bibitem{HoltzALS2012}
S.~Holtz, T.~Rohwedder, and R.~Schneider.
\newblock The alternating linear scheme for tensor optimization in the tensor
  train format.
\newblock {\em SIAM J. Sci. Comput.}, 34(2):683--713, March 2012.

\bibitem{Kolda09tensordecompositions}
T.~G. Kolda and B.~W. Bader.
\newblock Tensor decompositions and applications.
\newblock {\em SIAM REVIEW}, 51(3):455--500, 2009.

\bibitem{Mohlenkamp2013}
M.~J. Mohlenkamp.
\newblock {Musings on multilinear fitting.}
\newblock {\em Linear Algebra Appl.}, 438(2):834--852, 2013.

\bibitem{OR70}
J.~M. Ortega and W.~C. Rheinboldt.
\newblock {\em Iterative Solution of Nonlinear Equations in Several Variables}.
\newblock Society for Industrial Mathematics, 1970.

\bibitem{Oseledets2011}
I.~V. Oseledets.
\newblock Dmrg approach to fast linear algebra in the tt-format.
\newblock {\em Comput. Meth. in Appl. Math.}, 11(3):382--393, 2011.

\bibitem{OseledetsDolgov2012}
I.~V. Oseledets and S.~V. Dolgov.
\newblock Solution of linear systems and matrix inversion in the tt-format.
\newblock {\em SIAM J. Scientific Computing}, 34(5), 2012.

\bibitem{UschmajewALS2012}
A.~Uschmajew.
\newblock Local convergence of the alternating least squares algorithm for
  canonical tensor approximation.
\newblock {\em SIAM Journal on Matrix Analysis and Applications},
  33(2):639--652, 2012.

\bibitem{Uschmajew:2014}
Andr{\'e} Uschmajew.
\newblock A new convergence proof for the high-order power method and
  generalizations.
\newblock July 2014.

\bibitem{Wang2014}
L.~Wang and M.~Chu.
\newblock On the global convergence of the alternating least squares method for
  rank-one approximation to generic tensors.
\newblock {\em SIAM Journal on Matrix Analysis and Applications},
  35(3):1058--1072, 2014.

\bibitem{ZhangGolub2001}
Tong Zhang and Gene~H. Golub.
\newblock Rank-one approximation to high order tensors.
\newblock {\em SIAM J. Matrix Anal. Appl.}, 23(2):534--550, February 2001.

\end{thebibliography}

\end{document}